\documentclass[]{amsart}
\usepackage[left=2.7cm,right=2.7cm,top=3.5cm,bottom=3cm]{geometry}
\usepackage[english]{babel}
\usepackage{pifont}
\usepackage[latin1]{inputenc}
\usepackage[notref,notcite]{
}
\usepackage{amsmath}
\usepackage{amsfonts}
\usepackage{amsthm}
\usepackage{amscd}
\usepackage{amssymb}
\usepackage[all]{xy}
\usepackage{graphicx}
\usepackage[usenames,dvipsnames]{color}
\usepackage{mathrsfs}
\usepackage{caption}
\usepackage{tikz-cd} 
\usepackage{braket}
\usepackage{mathrsfs}
\usepackage[colorlinks = true,
linkcolor = black,
urlcolor = black,
bookmarksopen = true, 
backref = page]{hyperref}

\theoremstyle{plain}
\newtheorem{thm}{Theorem}[section]
\newtheorem{cor}[thm]{Corollary}
\newtheorem{lem}[thm]{Lemma}
\newtheorem{prop}[thm]{Proposition}

\theoremstyle{definition}

\newtheorem{rmk}[thm]{Remark}

\numberwithin{equation}{section}

\renewcommand{\ge}{\geqslant}
\renewcommand{\le}{\leslant}

\newcommand{\field}[1]{\mathbb{#1}}
\newcommand{\Q}{\field{Q}}
\newcommand{\C}{\field{C}}

\newcommand{\R}{\field{R}}

\newcommand{\N}{\field{N}}
\newcommand{\Z}{\field{Z}}
\newcommand{\A}{\field{A}}

\newcommand{\p}{\field{P}}
\newcommand{\G}{\field{G}}

\newcommand{\X}{\field{X}}

\newcommand{\nr}{nr}

\newcommand{\card}{\ding{171}}

\DeclareMathOperator{\PGL}{PGL}
\DeclareMathOperator{\End}{End}
\DeclareMathOperator{\Aut}{Aut}

\DeclareMathOperator{\Lie}{Lie}

\DeclareMathOperator{\GL}{GL}
\DeclareMathOperator{\SL}{SL}
\DeclareMathOperator{\Hom}{Hom}

\DeclareMathOperator{\Spec}{Spec}

\DeclareMathOperator{\rk}{rk}

\DeclareMathOperator{\vol}{vol}

\DeclareMathOperator{\Ram}{Ram}
\DeclareMathOperator{\Gr}{Gr}

\DeclareMathOperator{\CM}{\scal^{\text{CM}}}
\DeclareMathOperator{\SCM}{\scal^{\text{SCM}}}
\DeclareMathOperator{\rec}{rec}
\DeclareMathOperator{\Rat}{Rat}
\DeclareMathOperator{\Stab}{Stab}

\DeclareMathOperator{\Emb}{Emb}
\DeclareMathOperator{\Opt}{Opt}
\DeclareMathOperator{\Cl}{Cl}

\DeclareMathOperator{\Pic}{Pic^\circ}

\newcommand{\cala}{\mathscr A}

\newcommand{\calc}{\mathcal C}
\newcommand{\ccal}{\mathscr C}

\newcommand{\calf}{\mathcal F}
\newcommand{\calg}{\mathcal G}
\newcommand{\calh}{\mathcal H}

\newcommand{\calm}{\mathcal M}

\newcommand{\ncal}{\mathscr N}
\newcommand{\calo}{\mathscr O}

\newcommand{\cals}{\mathscr S}
\newcommand{\scal}{\mathcal S}
\newcommand{\calt}{\mathscr T}
\newcommand{\calv}{\mathscr V}

\newcommand{\calx}{\mathscr X}
\newcommand{\calz}{\mathscr Z}

\newcommand{\gotb}{\mathfrak b}

\newtheorem{innercustomgeneric}{\customgenericname}
\providecommand{\customgenericname}{}
\newcommand{\newcustomtheorem}[2]{%
	\newenvironment{#1}[1]
	{%
		\renewcommand\customgenericname{#2}%
		\renewcommand\theinnercustomgeneric{##1}%
		\innercustomgeneric
	}
	{\endinnercustomgeneric}
}

\newcustomtheorem{customthm}{Theorem}

\renewcommand{\ge}{\geqslant}
\renewcommand{\le}{\leqslant}

\newcommand{\hooklongrightarrow}{\lhook\joinrel\longrightarrow}

\newcommand{\interior}[1]{%
	{\kern0pt#1}^{\mathrm{o}}%
}

\title{Equidistribution of CM points on a Shimura Curve modulo a ramified prime}
\author{Francesco Maria Saettone}
\address{Department of Mathematics, Ben-Gurion University of the Negev, Israel}
\email{saettone@post.bgu.ac.il}

\begin{document}

\begin{abstract}
We prove an equidistribution statement for the reduction of Galois orbits of CM points on the special fiber of a Shimura curve over a totally real field attached to some ramified primes. To do so we study the reduction of CM points in the special fiber and we use Ratner's theorem to obtain the desired equidistribution.
\end{abstract}

\maketitle

\tableofcontents

\section{Introduction}

In \cite{cv} Cornut and Vatsal proved a simultaneous equidistribution result for reductions of Galois orbits of CM points in a quaternionic Shimura curve over a totally real field $F$, under the condition that the quaternion algebra splits at every prime at which the reduction takes place. In this work we remove this condition and suppose instead that the quaternion algebra ramifies at the primes of reduction.
\\

Equidistribution results for CM points on Shimura curves have been an active area of research in the last two decades. In the archimedean setting, among others  Clozel--Ullmo \cite{cu} and Zhang \cite{sz3} showed the convergence of  delta measures of certain orbits of CM points to the hyperbolic probability measure in the $\C$-points of the Shimura curve, with noticeable applications to the Andr\'e-Oort conjecture.
Cornut and Vatsal in a series of papers \cite{cor}, \cite{vat} culminating with \cite{cv} and \cite{cv2}, proved equidistribution modulo a non-split prime in a CM extension of $F$ and applies them to establish Mazur's conjecture on the non-triviality of Heegner points. A variant of their equidistribution result for PEL Shimura varieties is proved in \cite{xz}. In particular, Cornut and Vatsal proved that Galois orbits of the reduction of CM points are uniformly distributed with respect to the counting probability measure on the finite set of supersingular points. The same result (over $F=\Q$) was obtained by different methods  in \cite{jk}, which we plan to generalize to totally real fields in the forthcoming paper \cite{fms}.
\\On the non-archimedean side, Herrero--Menares--Rivera-Letelier \cite{hmrl} \cite{hmrl2} described the
set of accumulation measures of CM points
on the Berkovich $j$-line over the $p$-adics, and by these equidistribution methods they show that for every finite set of prime numbers $S$, there are at most finitely many singular moduli that are $S$-units. Lastly, Disegni \cite{dis} studied the equidistribution problem for CM points in the Berkovich analytification of a Shimura curve over the $p$-adics. For other related works see the references in  \cite{dis}.
\\

Let $B$ be a quaternion algebra over $F$ satisfying one of the following conditions:
\begin{itemize}
	\item  there is a unique real place $v$ of $F$ such that $B_v=B\otimes F_v$ is split\footnote{I.e., $B_v\simeq M_2(F_v)$.}, i.e., $B$ is $\mathit{indefinite}$;
	
	\item for every real place $B_v$ is non-split, i.e., $B$ is $\mathit{definite}$.
\end{itemize}
Fix $v$, a finite place of $F$ such that $B_v=B\otimes F_v$ is ramified. Let $K$ be a CM extension of $F$, so that to $B$ and $K$ we associate CM points. In this introduction we focus on the indefinite case, so that CM points inhabit the Shimura curves associated to $B$, while in the second case the definite Shimura set is a finite set. This situation, a priori rather different, is dealt with in the Appendix, since all the main proofs follow the same line of the (to our opinion, more interesting) indefinite case. We point out that in the setting of Shimura curves CM points are sometimes called Heegner points, especially when dealing with modular curves. We follow the terminology of CM point to differentiate between the points on the Shimura curve and the corresponding points on an elliptic curve, which we do not discuss.
\\

By the moduli interpretation of the local integral model at $v$ of the Shimura curve with maximal level structure, every point $x$ in the special fiber parametrizes a formal $\calo_{B_v}$-module, so that we denote by $\End(x)$ its endomorphism ring.

Let $\mu_{\text{ram}}$ be the normalized counting measure on the singular locus $\{s_1,...,s_h\}$ of special fiber of the Shimura curve given by
\[
\mu_{\text{ram}}(s_i)=\dfrac{w(s_i)}{\sum_{j=1} ^{h}w(s_j)}
\]
where $w(s_i)=\#\End(s_i)^\times/2$ is the $\mathit{weight}$ of $s_i$. 

On the other hand, let $\mu_{\text{in}}$ be the normalized counting measure on  the set of irreducible component $\{c_1,...,c_k\}$, given by
\[
\mu_{\text{in}}(c_i)=\dfrac{w(c_i)}{\sum_{j=1}^h w(c_j)}
\]
where  $w(c_i)=\#\End(c_i)^\times/2$ is the $\mathit{weight}$ of $c_i$. Here by $\End(c_i)$ we mean the endomorphisms of the formal group attached to the reduction of a CM point, which lands on the connected component $c_i$.

The following theorem restates the main result of this paper as in Theorem \ref{mainthm} .

\begin{customthm}{A}
	Let $K$ be fixed. The reductions of the Galois orbits of CM points are equidistributed for the conductors going $v$-adically to infinity\footnote{I.e., the $v$-adic valuation of the conductor varies.}:
	 \begin{itemize}
	 	 \item  in the singular locus of the special fiber of the Shimura curve with respect to $\mu_{\text{ram}}$ for $v$ ramified in $K$;
	 	
	 	\item  in the set of irreducible components of the special fiber of the Shimura curve with respect to $\mu_{\text{in}}$, for $v$ inert in $K$.
	 \end{itemize}
\end{customthm}

\begin{figure}
	\centering
	\includegraphics[width=0.7\linewidth]{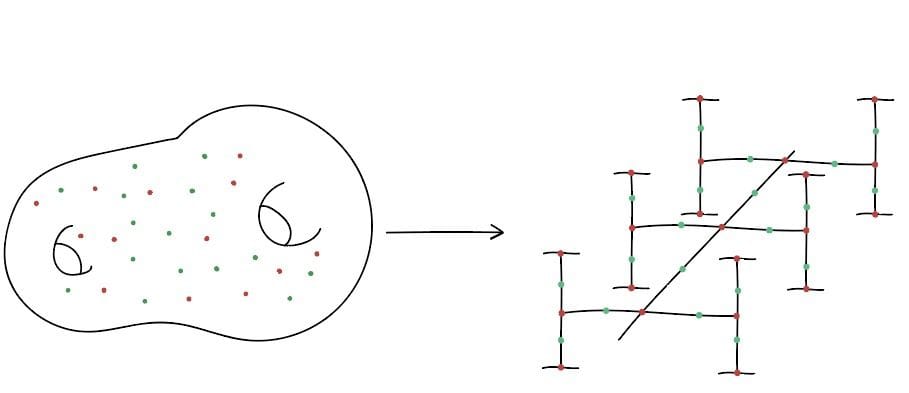}
	\caption{An illustration (drawn by Francesco Beccuti) of our main results, for a Shimura curve of genus $2$ over the $2$-adic numbers: the reduction map send the CM point to a red point for $v$ is ramified, and to a green one for $v$ inert.}
	\label{fig:reductiondrawingarxiv}
\end{figure}

The ramification of the quaternion algebra implies the existence of a $v$-adic uniformisation of Shimura curves by the Drinfeld upper half plane, and not by Lubin-Tate space, in sharp contrast with the unramified case. This is indeed one major difference with Cornut-Vatsal \cite{cv}, since there is a substantial modification of the moduli interpretation of these spaces.

This difference thus reverberates also in the in the study of the reduction of CM point in the special fiber. To do so, we exploit several interpretations of Drinfeld space (see \cite{bc}), and via Dieudonn\'e modules we relate them to (local) Ribet bimodules. We thus establish that the reduction at a ramified prime of CM point reduces to a singular (superspecial in the moduli interpretation) point in the special fiber, while at an inert prime it reduces to a smooth point of the special fiber. This result is perhaps well-known to experts, but we could not find a complete proof in the literature.
The reduction of CM points was studied, for $F=\Q$ in Ribet's beautiful paper \cite{rib}, and by the same methods an equidistribution statement is claimed in \cite{mol}. We point out that that equidistribution claim requires the discriminant of $K$ to vary, while, as we follow \cite{cv}, we make vary the conductor $v$-adically in a fixed CM field. 
 Moreover, another difference is that neither \cite{jk} nor \cite{mol} state any simultaneous equidistribution. 
 \\
 
The proof of our equidistribution theorem closely follows the strategy of \cite{cv}, whose main ingredient is the celebrated Ratner's theorem. For this result, we describe the sets of (special) CM points, their reductions, the connected components of the local model of the Shimura curve and the irreducible components of its special fibers as adelic double quotients as in \cite{sz} and \cite{sz2}. To do so, we  make an essential use of Cerednik-Drinfeld uniformization. In addition, Cornut--Vatsal focused only on the equidistribution on the set of supersingular point, i.e., they did not study the reduction by a non-split prime.

Moreover, we mostly focus on the case of  $v$ ramified in $K$. In fact, although geometrically the ramification and the inertia of $v$ produce two significantly different scenarios, the technical steps towards the application of Ratner's theorem proceed in parallel, leaving the impression that it is more natural a single description of this procedure. For the very same reason, we relegate a discussion of the case of a $\mathit{definite}$ quaternion algebra to the Appendix. 
\\

In conclusion, we mention that as an application of this equidistribution result one can obtain the proof of Mazur's conjecture in \cite{cv2} (sharpened in \cite{yzz}) in the ramified case.

\subsection{Notation and Conventions}
We list some notations used throughout this work:
\begin{itemize}
	\item we denote by $\C_v$  the completion of an algebraic closure of $F_v$, for $F$ a totally real number field
	\item we denote by $\mathbb{H}$ the classical Hamilton's quaternions
	\item $\A$ denotes the ring of adeles over $\Q$, and $\A_f$ the finite adeles
	\item for a quaternion algebra $B$ over $F$, we denote by $\Ram(B)$, the set of places of $F$ where $B$ ramifies; the same symbol, with $f$ or $\infty$ subscript, denotes the restriction to finite or archimedean places respectively
	\item for a set $S$, we denote by $\mathbf{1}_S$ the characteristic function of $S$
	\item for a local field $F$, we denote by $\breve{F}$ the completion of the maximal unramified extension
	\item for a field $K$, let $G^{ab}_K$ denote the abelianization of the absolute Galois group of $K$
	\item for a set $X$, we denote by $\calc(X,\C)$ the set of continuous functions on it, and in this context the superscript $*$ is inteded to be the $\C$-linear dual.
	
\end{itemize}

\subsubsection*{Acknowledgments}

It is a pleasure to thank Professor Daniel Disegni for his guidance with my doctoral thesis, which this article is part of. I am also grateful to Zev Rosengarten for illuminating discussions culminating with the communication of \cite{ros}. This work was supported by ISF grant 1963/20 and BSF grant 2018250.

\section{Shimura Curves}

\subsection{Quaternion Algebras}
Let $F$ be a totally real number field of degree $d$ over $\Q$. Let $B$ be a quaternion algebra $\mathit{indefinite}$ exactly at one archimedean place of $F$, denoted by $\tau_1\colon F\hookrightarrow \R$. 
\\Let $\mathbb{S}$ be the Deligne torus, namely the real algebraic group of invertible elements in the $\R$-algebra $\C$, i.e., $\mathbb{S}=\text{Res}_{\C/\R}\G_m$. Hence $\mathbb{S}(\R)\simeq\C^\times$ and $\mathbb{S}_\C\simeq \G_{m,\C}\times\G_{m,\C}$.
\\Consider the reductive group $G$ over $\Q$ whose functor of points, for a commutative $\Q$-algebra $A$, is given by 
\[
G(A)=(B\otimes_\Q A)^\times.
\]
Then $G=\text{Res}_{F/\Q}B^\times$ and $h_0\colon\mathbb{S}\rightarrow G_\R$ is a real embedding with trivial coordinates at $\tau_i$ for $i\ge 2$.  
Indeed, $G_\R(\R)=(B\otimes_\Q\R)^\times$, where $B\otimes\R=\prod_{\tau}B_\tau=B_{\tau_1}\times\prod_{i=2}^{d}B_{\tau_i}\simeq M_2(\R)\times \mathbb{H}^{d-1}$.

Consider the conjugacy class of $h_0$ under the action of $G(\R)$, which is isomorphic to the Poincar\'e ``half'' plane $\mathscr{H}=\C-\R$. For any compact open subgroup $U$ of $G(\A_f)$, the Shimura curve $\scal_U$ over $F$ is the canonical model of
\begin{equation}\label{shi}
\scal^{an}_U:=G(\Q)\backslash \mathscr{H}\times G(\A_f)/ U \cup \{\text{cusps}\}
\end{equation}
where the set $\{\text{cusps}\}$ is non-empty if and only if $B=M_2(\Q)$. We remark that the cusps play no role in this work.

Regarding the moduli interpretation of these curves, we briefly recall the main facts and we refer to \cite{sz} and \cite{sz2} for more details.

If $F=\Q$, then $\scal_U$ is a moduli space of QM abelian surfaces\footnote{I.e., with quaternionic multiplication.}, and the extension of the moduli problem to $\Z$ allows to define canonical integral models. On the other hand, when $d>1$, there is no convenient parametrization of abelian varieties. Nonetheless, $\scal_U$ has a finite map to another Shimura curve $\scal_{U'}$ carrying the desired parametrization. More specifically, the curve $\scal_{U'}$ can be viewed as a certain moduli space of QM abelian varieties with a polarization and a rigidification \cite[Prop.1.1.5]{sz2}.

Following \cite{bc}, in the next sections we will describe an integral regular model $\cals_U$ over $\Spec\calo_F$, in the case of $U_v$ maximal at $v$, for $v\in \Ram(B)$.

\subsubsection{Ramification Setting}\label{rm}

From now on, let $K$ be a CM quadratic extension of $F$, and let $B$ denote a quaternion algebra over $F$ split by $K$, unless stated differently.
\\We fix an embedding $K\hookrightarrow B$ over $F$. Let $S$ be a finite set of finite places of $F$ such that, for all $v\in S$,

\begin{itemize}
	\item $B$ ramifies at $v$;
	\item $ \# \Ram_f(B) - \# S + d$ is even;
	\item $v$ is inert or ramified in $K$.
\end{itemize}

The first two conditions implies (by \cite[Thm.3.1, p.74]{vign} or \cite[Prop.14.2.7]{voi}) that there exists a  totally\footnote{That is, all archimedean places are ramified in $B_S$.} definite quaternion algebra $B_S$ over $F$ such that $\Ram_f(B_S)=\Ram_f(B)-S$.
\\Note that the last condition  is implied by the first one, since our hypotheses do not allow the split case, i.e., if $B_v$ ramified then $K_v$ is not  split for $F_v$. Indeed $B_v$, while it contains every quadratic extension of $F_v$, it cannot contain $F_v\oplus F_v$, since otherwise the action on $F_v^2$ given by multiplication would identify $B_v$ with $2$ by $2$ matrices, leading to a contradiction with the hypothesis that $B$ is ramified.

\subsection{Local Integral Models}

\subsubsection{The Formal Scheme $\widehat{\Omega}$}
Let $F_v$ be a finite extension of $\Q_p$ with residue field $k$. We begin by recalling the rigid-analytic Drinfeld plane $\Omega$ over $F_v$ whose $\C_v$-points are given by
\[
\p^1(\C_v)-\p^1(F_v).
\]
This is the generic fiber of the formal scheme  whose construction we summarize below.
\\

Let $V$ be a vector space of dimension $2$ over $F_v$. The $\mathit{Bruhat}$-$\mathit{Tits}$ tree $\calt$
for $\PGL_2(V)$ is defined having vertices the homothety classes of $\calo_{F_v}$-lattices (i.e., free $\calo_{F_v}$-modules of $\rk$-$2$) on $V$. Two such classes $[\Lambda]$, $[\Lambda']$ are connected
by an edge if there exists $\Lambda''\in [\Lambda']$ such that $\Lambda''\subseteq \Lambda$ and $\Lambda/\Lambda''$ has length $1$ as $\calo_{F_v}$-module.

Let $s=[\Lambda]$ denote the homothety class of a $\calo_{F_v}$-lattice $\Lambda$  in a $2$-dimensional extension  of $F_v$. We denote by $\p_s$ the projective line $\p^1(\Lambda)=\text{Proj}(\text{Sym}(\Lambda))$ over $\calo_{F_v}$, where $\text{Sym}(\Lambda)$ is the usual symmetric algebra. Note that all the $\p_s$'s are birationally isomorphic.

Consider the open subset of $\p_s$
\[
\Omega_v:=\p_s-\p_{s,k}(k),
\]
and its completion $\widehat{\Omega}_s$ along the special fiber.

\begin{rmk}
	In other words, the formal model $\widehat{\Omega}_s$ arises from successive blow-ups of rational points on the special fibers of $\p_v$.
\end{rmk}

It is easy to see that, if $s$ and $s'$ are two adjacent vertices, then $s'$ defines a $k$-rational point of $\p_{s,k}$. Indeed we have $\Lambda\rightarrow \Lambda/\Lambda'\simeq k$. We glue $\p_s$ and $\p_{s'}$ via this $k$-rational point and we denote the resulting $\calo_{F_v}$-scheme by $\p_{[s,s']}$.
Thus, for an edge $[s,s']$, we define
\[
\Omega_{[s,s']}:=\p_{[s,s']}-(\p_{s',k}(k)\cup \p_{s,k}(k)) 
\]

and its completion $\widehat{\Omega}_{[s,s']}$ along the special fiber. 
The generalization of this procedure to finite subtrees $T\subset \calt$ leads to the regular formal scheme
\[
\widehat{\Omega}=\bigcup_T\widehat{\Omega}_T
\]
defined over $\calo_{F_v}$. It has semi-stable reduction.

\subsubsection{Moduli Interpretations of the Drinfeld Plane}\label{moduli}

The Drinfeld half plane admits a moduli interpretation, which now we explain. 
\\

Let $A$ be a ring. Consider the functor $\calc$ from the category $C$ of $A$-algebras complete with respect to an ideal of $A$ to the category of sets sending $R$ to the set of $n$-tuples of topologically nilpotent elements of $R$. An $n$-dimensional $\mathit{formal}$ ($\mathit{Lie}$) variety over $A$ is a functor $\calv$ from $C$ to sets which is isomorphic to $\calc$. Its tangent space
$\Lie(\calv)$ is the functor $\calv$ evaluated on the space of dual numbers. 
From now on, the we will deal with $n=1$ only. 
\\We thus define a $\mathit{formal}$ group over $A$ to be a group object in the category of formal varieties over $A$.  Let $\calo$ be a commutative ring and let $A$ be an $\calo$-algebra. A formal $\calo$-module over $A$ is a  formal group $X$ together with a ring homomorphism $\calo\rightarrow \End(X)$ whose derivative $\calo\rightarrow\End(\Lie(X))=A$ is the structure map $\calo \rightarrow A$. 

We recall a morphism $f\colon X\rightarrow X'$ between two formal groups over $A$ is an $\mathit{isogeny}$ if it is surjective and its kernel is a finite locally free
group scheme. For $A=\Z_p$, the kernel of $f$ is of rank a power of $p$ and such a power is called $\mathit{height}$ of the isogeny and denoted by $\text{ht}(f)$.
\\For two formal $\calo_{B_v}$-modules $X$ and $X'$, a morphism $\varphi\in\Hom_{\calo_{B_v}}(X,X')$ is a $\mathit{quasi}$-$\mathit{isogeny}$  if there exists a $m\ge 0$ such that $\varpi^m\varphi$ is a $\calo_{B_v}$-isogeny\footnote{I.e., an isogeny which is linear with respect to the $\calo_{B_v}$-action.}, for $\varpi$ a prime of $\calo_{K_v}$.
The height of a quasi-isogeny can now be defined as $\text{ht}(\varpi^m\varphi)-\text{ht}(\varpi^m)$.

Let $L_v$ be the unique quadratic unramified extension of $F_v$, and fix an embeding $F_v\hookrightarrow B_v$. Let $\varpi_{F_v}$ be a uniformizer of $F_v$. Pick an element $\Pi$ of $\calo_{B_v}$ such that $\Pi^2=\varpi_{F_v}$ and $\Pi a=a^\sigma \Pi$, for $\sigma$ the conjugation of $L_v$ over $F_v$ and $a\in\calo_{L_v}$.
Thanks to the grading induced by $\Pi$, we can write the homogeneous components of the tangent space of a formal module $X$ as follows
\begin{align*}
	&\Lie(X)_0=\{x\in\Lie(X) : a.x=ax, \;\forall a\in\calo_{L_v}\}
	\\
	&\Lie(X)_1=\{x\in\Lie(X) : a.x=a^{\sigma}x,\;\forall a\in\calo_{L_v}\}.
\end{align*}

We say that a formal module is $\mathit{special}$ if both the homogeneous components of the tangent space have dimension $1$.

Lastly, we denote by $\X$  the framing object, i.e., the unique, up to isogeny, formal module of height $4$ over $\overline{k}$.

Let $\text{Nilp}(\breve{\calo}_{F_v})$ be the category of $\breve{\calo}_{F_v}$-algebras in which the image of the uniformizer $\varpi_{F_v}$ of $\calo_{F_v}$ is nilpotent. Consider the functor $\calg$ on $\text{Nilp}$ assigning to $R$ the set of isomorphism classes of pairs $(X,\varrho)$ consisting of

\begin{itemize}
	\item a formal special $\calo_{B_v}$-module $X$ of height $4$ over $R$
	\item a height zero quasi-isogeny $\varrho\colon \X_{R/\varpi_{F_v}}\rightarrow X_{R/\varpi_{F_v}}$.
\end{itemize}

 Drinfeld proved that this functor is represented by the formal $\breve{\calo}_{F_v}$-scheme
\[
\widehat{\Omega}\widehat{\otimes}_{\calo_{F_v}}\breve{\calo}_{F_v}.
\]

 For a detailed discussion of the proof we refer to \cite[p.107]{bc} and to \cite{bz}.
 \\
 
 We now introduce two other equivalent moduli interpretations in terms of (semi-)linear objects.
 
 For a $\calo_{F_v}$-algebra $A$, consider the free rank-$2$ $A$-algebra $A[\Pi]=A[X]/(X^2-\varpi_{F_v})$. It has a  $\Z/2\Z$-grading such that $\deg(a)=0$ for $a\in A$ and $\deg(\Pi)=1$.

 Let $\calf$ be a functor sending an object  $A$ of $\text{Nilp}(\calo_{F_v})$ to the isomorphism class of
 	\[
 	(\eta, T, u, r)
 	\]
  where:
 	
 		\begin{itemize}
 		\item  $\eta$ is a $\Z/2\Z$-graded sheaf of flat $\calo_{F_v}[\Pi]$-modules constructible over $S=\Spec A$ of locally finite presentation, and  the homogeneous components $\eta_0$, $\eta_1$ are sheaves of flat $\calo_{F_v}$-modules whose fibers are free of rank-$2$;
 		
 		\item  $T$ is a $\Z/2\Z$-graded sheaf of $\calo_S$-modules whose the homogeneous components $T_0$, $T_1$ are line bundles on $S$;
 		
 		\item $u\in\Hom_{\calo_S[\Pi]}(\eta,T)$ is a map of degree $0$ such that $u\otimes\calo_S\colon\eta\otimes\calo_S\twoheadrightarrow T$;
 		\item $r$ is a rigidification $r\colon \underline{F_v}^2\simeq \eta_0\otimes_\calo\underline{F_v}$ where $\underline{F_v}$ is the constant scheme.
 	\end{itemize}

In the (in)famous paper \cite{drin}, Drinfeld proved that the functor $\calf\otimes \breve{\calo}_{F_v}$ is isomorphic to $\calg$, hence represented by $\widehat{\Omega}\widehat{\otimes}\breve{\calo}_{F_v}$. We refer to \cite{bc} for a more detailed explanation of Drinfeld's powerful ideas.
\\

We now review some results on the covariant Dieudonn\'e-Cartier theory, and we sketch its relation with the functor $\calf$.

For a $\calo_{F_v}$-algebra $\calo$, let $W$ denote the Witt vectors of $\calo$ endowed with the $q$-Frobenius $\sigma$. Consider the Frobenius $F$ and Verschiebung $V$ operators on $W$, characterized by the relations $FV=VF=\varpi_{F_v}$, $Fw=w^\sigma F$, $wV=Vw^\sigma$, for $w\in W$. 
The $\mathit{Cartier}$ ring $E_k$ is defined as the $V$-adic completion of the ring of non-commutative polynomials $W[F,V]$. We call a ($\calo_{F_v}$-)$\mathit{Dieudonne}$ module  a left $E_k$-module $D$ such that 
\begin{itemize}
	\item $D/VD$ is free of finite rank;
    \item $V$ is injective;
    \item $D$ is $V$-adically complete and separated.
\end{itemize}

 The fundamental result of Cartier-Dieudonn\'e theory states there is an equivalence of categories between formal  $\calo_{B_v}$-modules over $\calo_{F_v}$ and $\calo_{F_v}[\Pi]$-Dieudonn\'e  modules (see, for instance, \cite[Thm.p.74]{bc}). 
 \\Note that under this equivalence the height of the formal modules is preserved in the rank of the Dieudonn\'e module.
 \\Moreover we say that $D$ is $\mathit{special}$ if both $D_0/VD_1$ and $D_1/VD_0$ are both free of rank $1$, where $D_i$'s are induced by the $\Z/2\Z$-grading. Via this equivalence, $D$ is special if and only if the corresponding formal module $X$ is such. 
 
 Drinfeld thus obtained in \cite{drin} the correspondence between the triples $(\eta, T, u)$ and Dieudonn\'e modules of $X$, which we summarize with the following bijections
 \begin{equation}\label{Dieudonne}
 \Set{\begin{array}{c}
 		\text{Special formal modules}\\
 		\text{of height $4$}  \\
 \end{array}}
 \longleftrightarrow
 \Set{
 	\begin{array}{c}
 		\text{Special Dieudonn\'e}
 		\\
 		\text{modules of $\rk$-4}
 \end{array}}
 \longleftrightarrow
 \Set{\begin{array}{c}
 		(\eta, T, u)
 \end{array}}
 \end{equation}
 
 where the sets are considered up to their respective isomorphisms.
 
Let us briefly explain these equivalences. From what discussed above, we have that for $D$ the Dieudonn\'e module of a formal module $X$, we obtain that $D/VD$ corresponds to $\Lie(X)$, which is associated to $T$, with the compatibility of their various actions and gradings. Note that $T$ is required to be a sheaf, and this is achieved by applying (the inverse of) Hartshorne's classical construction of a sheaf of modules.

On the other hand, the construction of the counterpart of $\eta$ is rather sophisticated, and here we just recall the main definitions following \cite[pp.77-86]{bc}.

The $\mathit{modified}$ Dieudonn\'e module $N(D)$ is the cokernel of the map
\[
D\hookrightarrow D\oplus D^\sigma,\;d\mapsto(Vd,-\Pi d) .
\]
A key technical point of this construction is replacing the action of $F$: let us forget the action of the Frobenius, so that $N(D)$ inherits a $W[V,\Pi]$-action. Let $D$ be special, and consider the map $\lambda_D\colon N(D)\rightarrow N(D)$ defined as $[d,d']\mapsto\Pi d+Vd'$. There exists a unique map $L_D\colon D\rightarrow N(D)$ such that $\lambda_D\circ L_D= F$ by \cite[Prop.3.8]{bc}. We thus obtain the $\mathit{modified}$ Frobenius 
\[
\varphi_D\colon N(D)\rightarrow N(D),\;[d,d']\mapsto L_D(d)+[d',0]
\]
which is a $\calo_{F_v}[\Pi]$-linear map. Finally, to every special Dieudonn\'e $\calo_{F_v}[\Pi]$-module $D$ we associate a $\calo_{F_v}[\Pi]$-module $\eta_D$ such that
$$\eta_D=N(D)^{\varphi_D=1}=\{z\in N(D): \varphi_D(z)=z \}.$$
We also define $u_D\colon\eta_D\rightarrow D/VD$ to be the $\calo_{F_v}[\Pi]$-linear, $\deg$-0 map defined by composing the natural maps $\eta_D\hookrightarrow N(D)$ and $N(D)\twoheadrightarrow D/VD$.

Once again, by applying Hatshorne's classical construction of a sheaf of modules,  we obtain the sheaf  $\tilde{\eta}_D$, which we keep referring to, by slight abuse of notation, as $\eta_D$.

\subsubsection{Local Ribet Bimodules}\label{lRb}
In this section we exploit the beautiful work of Ribet \cite{rib}, and we connect it with Drinfeld moduli interpretation discussed before.

Fix a place $v$ of $F$ which ramifies both in $K$ and in $B$.
A local  $(\calo_{B_v},\calo_{B_v})$-$\mathit{bimodule}$ $\calm$ is a free module of finite rank over $\calo_{F_v}$ furnished with the structure of left and right projective $\calo_{B_v}$-module. Once we denote by $\gotb$ the maximal ideal of $\calo_{B_v}$, we say that $\calm$ is  $v$-$\mathit{admissible}$ if 
\[
\gotb\calm=\calm\gotb.
\]

For $\calm$ $v$-admissible, the classification result \cite[Thm.1.2]{rib} shows that 
\[
\calm\simeq \calo_{B_v}^{r_v}\oplus\gotb^{s_v}
\]
where $\calo_{B_v}$ and $\gotb$ are regarded as bimodules via the natural multiplication of $\calo_{B_v}$ on itself and on $\gotb$. For such a decomposition, we say that $\calm$ is of $\mathit{type}$ $(r_v,s_v)$, and we impose that $r_v+s_v=2$. Here we deal with bimodules of $\rk$-$8$. 

Note that since $\calm$ is a right-$\calo_{B_v}$-module, free of rank $n>4$ over $\calo_{F_v}$, then $\calm$ is free as a right $\calo_{B_v}$-module (see \cite[p.38]{rib}). After fixing an isomorphism, we have the natural identification $\End_{\calo_{B_v}}(\calm)=M_2(\calo_{B_v})$, so that the left $\calo_{B_v}$-module structure of $\calm$ can be given by
\[
f\colon\calo_{B_v}\hooklongrightarrow M_2(\calo_{B_v}).
\]
The projectivity of $\calm$ imples that $f$ is optimal\footnote{Whose definition is recalled in the Appendix, Section \ref{grosspoints}.}, and  by \cite[Thms.1.4,1.6]{rib} the isomorphism class of $\calm$ is determined by the $\GL_2(\calo_{B_v})$-conjugacy of $f$.  This means that we have the following bijections
\begin{equation}\label{rbim}
\Set{
	\begin{array}{c}
		\text{ring homorphisms}
		\\
		\text{$f\colon\calo_{B_v}\rightarrow M_2(\calo_{B_v})$}\\
		\text{up to $\GL_2(\calo_{B_v})$-conjugacy}\\
\end{array}}
\longleftrightarrow
\Set{\begin{array}{c}
	\text{$(\calo_{B_v},\calo_{B_v})$-bimodules}\\
	\text{free or $\rk$-8 over $\calo_{F_v}$}\\
	\text{up to isomorphism}\\
\end{array}}.
\end{equation}

Moreover, in term of $f$, we obtain that  $v$-admissibility means that $f(\gotb)\subseteq M_2(\gotb)$. Let $\varpi_{K_v}$ be a uniformizer of $\calo_{K_v}$. By the ramification hypothesis,  we have that $\gotb$ is generated by $\varpi_v$, i.e., $\gotb=\varpi_{K_v}\calo_{K_v}$. It immediately follows that $f(\gotb)$ consists of the matrices whose coefficients are divisible by $\varpi_{K_v}$.

\begin{lem}\label{eichorder}
	Let $f\colon\calo_{B_v}\hooklongrightarrow M_2(\calo_{B_v})$ be such that $f(\gotb)\subseteq M_2(\gotb)$ corresponding to a bimodule of type $(1,1)$. Then the commutant of $f(\calo_{B_v})$ in $M_2(\calo_{B_v})$ is an Eichler order of level $\varpi_v$ in $M_2(\calo_{F_v})$.
\end{lem}
\begin{proof}
	This follows from the previous discussion and \cite[Cor.1.5]{rib}.
\end{proof}

\begin{rmk}
	In \cite[Def.5.8 p.94]{bc} a triple $(\eta, T, u)$ is defined  $\mathit{admissible}$ if, for $i\in\Z/2\Z$,
	\begin{itemize}
		\item if $S_i\subset S$ is the zero locus of $\Pi\colon T_i\rightarrow T_{i+1}$, then $\eta_{i}|_{S_i}$ is a constant sheaf with fibers isomorphic to $\calo^2$
		\item for every geometric point $x$ of $S$, denote $T(x):=T\otimes k(x)$. Then $u$ induces the injection $$\eta_x/\Pi\eta_x\hookrightarrow T(x)/\Pi T(x).$$
		
	\end{itemize}
	We remark that the case of our interest, i.e., when such such data come from the Dieudonn\'e module of a formal module, these two conditions are automatically satisfied \cite[Prop. 5.10]{bc}.
	
	Lastly, we underline that this terminology is  $\mathit{not}$ related to the concept of $v$-$\mathit{admissible}$ Ribet bimodule. This is particularly clear from a geometric perspective.
\end{rmk}

\subsubsection{Cerednik-Drinfeld Uniformization}

  Let $\widehat{\cals}$ denote the formal completion along the special fiber over $v$. Let $B'$ be the quaternion algebra obtained by changing invariants at $v$ and $\tau$, i.e., $B'$ is definite, unramified at $v$, and $B'_\ell\simeq B_\ell$ for all $\ell\neq v, \tau$. Let  $G'$ be the reductive group  whose functor of points is $G'(A)=(B'\otimes_\Q A)^\times$.   
  \\Let $U^v$ a compact open subgroup of $G(\A_f^v)$  small enough and assume that $U_v$ is maximal, i.e., $U_v\simeq \calo_{B_v}^\times$. Then Cerednik-Drinfeld uniformisation (see \cite{drin}, \cite[p.241]{yzz}, \cite[p.44]{sz2} ) consists of the following isomorphisms
    
    \begin{equation}\label{cdu}
    	\begin{aligned}
    &\scal_U^{an,\bullet}\simeq G(\Q)\backslash (\Omega^\bullet\widehat{\otimes}\breve{F}_v)\times\Z\times G'(\A_f^v)/ U^v
	\\
	&\widehat{\cals}^{\bullet}_U\simeq G(\Q) \backslash \widehat{\Omega}^\bullet\widehat{\otimes}\breve{\calo}_{F_v}\times \Z\times G'(\A_f^v)/U^v
	\end{aligned}
	\end{equation}
 
 for the superscript $\bullet\in \{\emptyset, \;' \}$, where $\bullet=\emptyset$ means the absence of any index, while in the other case we define 
 \begin{equation*}
 	 \scal_U^{an,\prime}=\scal_U^{an}\widehat{\otimes} \breve{K}_v,\;\;\;\widehat{\cals}'_U=\widehat{\cals}_U\widehat{\otimes}\breve{\calo}_{K_v}
 \end{equation*}
 and
 \begin{equation*}
 \Omega'=\Omega\widehat{\otimes}\breve{K}_v.
 \end{equation*}
If $v$ is unramified in $K$, then we can consider $\bullet=\emptyset$. Since $K_v$ is unramified over $F_v$, it follows that $\breve{K}_v=\breve{F}_v$.
\noindent
\\If $v$ is ramified in $K$, then $\widehat{\Omega}'$ consists of the resolution of all double points on the special fiber of $\widehat{\Omega}\widehat{\otimes}\breve{\calo}_{K_v}$, and this will imply some changes in Section \ref{dualgr}.

In  both case, we conclude that studying the irreducible components of the special fiber of the Shimura curve can be reduced through (\ref{cdu}) to the special fiber  $\widehat{\Omega}^\bullet\otimes_k \bar{k}$.

 \subsubsection{Special Fiber of $\widehat{\Omega}^\bullet$}\label{dualgr}
 Let us first analyze the case $\bullet=\emptyset$.
 \noindent
 \\The combinatorial structure of the special fiber of $\widehat{\Omega}$ is recorded by its dual graph.
 We recall that the dual graph of a semi-stable\footnote{I.e., geometrically
 	reduced with every singularity an ordinary double point.} curve is defined by taking as vertices its irreducible components, which are connected by an edge whenever the two corresponding components intersect in  a subscheme of codimension $1$. In particular, we focus on the dual graph of the special fiber of $\widehat{\Omega}$. 
 
 The irreducible components of $\widehat{\Omega}_{\bar{k}}$ are projective lines  (see \cite[p.145]{bc}). 

 Clearly, since all irreducible components of $\widehat{\Omega}_{\overline{k}}$ are $\p^1_{\overline{k}}$, hence smooth, there is no self intersection, i.e., the dual graph has no loops.
 
 \begin{lem}\label{dualgraph}
 	The dual graph of the special fiber of $\widehat{\Omega}$ is isomorphic to the Bruhat-Tits tree $\calt$. 
 	
 \end{lem}
 
 \begin{proof}

 	Consider the map from $\calt$ to the dual graph of $\widehat{\Omega}_{\overline{k}}$  defined respectively on vertices and edges by: 
 	\begin{equation*}
 		\begin{aligned}
 			& v\mapsto\p^1_v
 			\\
 			& [v,v']\mapsto (\Lambda\rightarrow \Lambda/\Lambda'\simeq \bar{k})
 		\end{aligned}
 	\end{equation*}
 	
 	where  $(\Lambda\rightarrow \Lambda/\Lambda'\simeq \overline{k})$ is the $\overline{k}$-rational point defined by $v'$ for $v,v'$ adjacent.
 	Thus in the dual tree $\p_v$ intersects $\p_{v'}$ if and only if $v$ and $v'$ are adjacent.
 	
 	It follows that the dual graph of $\widehat{\Omega}_{\overline{k}}$ is isomorphic to the dual tree  which consists  of projective lines over $\overline{k}$ intersecting each other in their $\overline{k}$-rational points.
 \end{proof}

\begin{rmk}
By Lemma \ref{dualgraph}  the configuration of the irriducible components of $\widehat{\Omega}\otimes\bar{k}$ is described by $\calt$. 
\\Moreover  Lemma \ref{dualgraph} also implies that the edges of $\calt$ correspond to the singular points in $\widehat{\Omega}_{\bar{k}}$.
\end{rmk}

On the other hand, in the case $\bullet=\;'$, the special fiber $\widehat{\Omega}'\otimes\bar{k}$ consists of the strict transforms of the irreducible components of $\widehat{\Omega}\otimes\bar{k}$. Through this, the exceptional divisors corresponds to double points in $\widehat{\Omega}$. Hence, we have the following bijection between  the exceptional divisors and the intersection points described above  
		\[
	\Set{\begin{array}{c}
			\text{exceptional divisors}\\
			\text{of\; $\widehat{\Omega}'\otimes\bar{k}$}  \\
	\end{array}}
	\longleftrightarrow
	\text{Sing}(\widehat{\cals}_k):=\Set{
		\begin{array}{c}
			\text{singular points}
			\\
			\text{of $\widehat{\cals}_k$ }
	\end{array}}.
	\]
	
Lastly, we conclude with an algebraic characterization of the singularities exploiting the moduli interpretation from Section \ref{moduli}.	Let $D$ be the Dieudonn\'e module of a formal module as in the bijection (\ref{Dieudonne}).
We say that $i\in\Z/2\Z$ is $\mathit{critical}$ index for $D$ if 
 
 $$\Pi\colon D_i/VD_{i-1}\rightarrow D_{i+1}/VD_i$$ 
 
 is zero, i.e., $\Pi D_i\subset VD_i$.

Since the tangent space\footnote{I.e., deformations over the dual numbers.} of $\calg$ is $2$-dimensional, then the irreducible components, i.e., projective lines over $\overline{k}$, intersect in the points with exactly two critical indexes.
 This means that they interect in a point representing $(\eta, T, u)$ with $\Pi\colon T_0=\Lie(X)_0\rightarrow T_1=\Lie(X)_1$ and $\Pi\colon \Lie(X)_1\rightarrow \Lie(X)_0$ are both the zero map, where we recall that $\Lie(X)_i=D_i/VD_{i-1}$. In other words, $\Pi D_i\subset VD_{i}$.

\subsection{Special CM Points}
The set of $F\otimes\A_f$-embeddings $\tau\colon K\otimes \A_f\hookrightarrow B\otimes\A_f$ is non empty (see, for instance \cite{sz2}). For any such $\tau$, the group $\tau(K^\times)$ is contained in $G(\A_f)$, and hence it acts on the Shimura curve $\scal_U$.
\\The scheme of $\mathit{CM}$ $\mathit{points}$ by $(K,\tau)$ is the fixed-point affine subscheme $\scal_U^{\tau(K^\times)}$. Let $ K_{\tau^{-1}(U)}$ denote the abelian extension of $K$ with Galois group $K^\times\backslash (K^\times\otimes \A_f)$. The theory of complex multiplication (as in \cite[Section 5.2]{sz2}) shows that $\scal_U^{\tau(K^\times)}$ is isomorphic to the $F$-scheme $\Spec K_{\tau^{-1}(U)}$. 
\\The CM ind-subscheme of $\scal_U$ is 
\[
\scal^{\text{CM}}_U=\bigcup_{(K,\tau)}\scal_U^{\tau(K^\times)}.
\]
The algebraicity of CM points is guaranteed by Shimura theory \cite{del}, which shows that they are defined over the maximal abelian extension of $K$.

By the moduli interpretation, for a CM point $z$ over the base-changed scheme $\CM\otimes F_v$, consider the corresponding formal $\calo_{B_v}$-module $X$. Then the $\calo_{F_v}$-order $\End(z)$ in $K_v$ is the ring of $\calo_{F_v}$-linear endomorphisms of $X$, and it has a well-known classifying description. Indeed we have, for a unique positive integer $c$,
\[
\End(z)=\calo_c:=\calo_{F_v}+\varpi_{v}^c\calo_{K_v},
\]
where $\varpi_v$ is the uniformizer of $\calo_{F_v}$ and $c$ is the ($v$-)$\mathit{conductor}$ of $z$.

In terms of double quotients, we have also the following description. Let $T$ be the $\Q$-rational torus in $G$, i.e., $T=\text{Res}_{K/\Q}(\mathbb{G}_m)$.

\begin{lem}\label{CM}
	The set of CM points is in bijection with
\begin{equation}
	G(\Q)\backslash G(\Q)h_0\times G(\A_f)/ U\simeq T(\Q)\backslash G(\A_f)/U.
\end{equation}
\end{lem}
\begin{proof}
	See \cite[Lemma 3.7, p.41]{cv}.
\end{proof}

We recall that the continuous, $G(\A_f)$-equivariant Galois action of $G_K^{ab}$ on $\scal^{\text{CM}}_U$ is given by multiplication of $T(\A_f)$ via Artin's reciprocity map $\rec_K$.
 Namely, we have that, for $x\in\scal^{\text{CM}}_U$ corresponding to $[g]\in G(\A_f)$,  and $\sigma=\rec_K(t)$, with $t\in T(\A_f)$,
\[
\sigma x=[tg]\in\scal^{\text{CM}}_U.
\]

In what follows, we will only need points in $\CM$ whose conductor at $v$ is maximal, i.e., $c=0$. In the literature they are called $\mathit{special}$ CM points (see \cite[p.261]{sz2}). We denote this set by $\scal^{\text{SCM}}_U$, whose description as a double quotient is given by the next Lemma.

\begin{lem}\label{sCM}
	The special CM points $\scal^{\text{SCM}}_U$ are in bijection with
	\begin{equation}\label{SCM}
		T(\Q)_0\backslash G'(\A_f^v)/U^v,
	\end{equation}
where $T(\Q)_0$ denotes the element whose component at $v$ has order $0$.
\end{lem}
\begin{proof}
	Every special CM point is represented by some $g\in G(\A_f)$ whose component at $v$ is in $T(F_v)\cdot U_v$. Note also that $T(\Q)_0$ can be viewed as the set of elements in $T(F)$ whose image in $T(F_v)$ is in $U_v$. Lastly, since $G(\A_f^v)\simeq G'(\A_f^v)$, we obtain the required bijection. 
\end{proof}

\begin{rmk}
	We remark that  special CM points can be related to $\mathit{good}$ CM point as defined in \cite[Def.1.6,  p.9]{cv2}. It is immediate that the first set is contained in the latter.
\end{rmk}

\subsection{Superspecial Points}

By Cerednik-Drinfeld uniformisation (\ref{cdu}) and Drinfeld moduli interpretation of $\widehat{\Omega}$, a geometric point $x$ in the special fiber of $\widehat{\cals}$ is said 
$\mathit{supersingular}$ if the corresponding formal $\calo_{B_v}$-module $X_k$ is isogenous to
\[
Y_k\oplus Y_k
\]
where $Y_k$  is a fixed formal $\calo_{F_v}$-module of height $2$ and dimension $1$ such that $\calo_{B_v}\simeq \End_{\calo_{F_v}}(Y_k)$.

\begin{lem}
 Every geometric point $x$ in $\widehat{\cals}_k$ is supersingular.
\end{lem}
\begin{proof}
	Consider the $v$-adic Tate module of the formal module $X_k$ corresponding to $x$, i.e.,
	\[
	T_v(X_k):=\varprojlim_n X_k[\varpi_v^n]
	\]
	where $X_k[\varpi_v^n]$ denotes the $\varpi_v^n$-torsion points of $X_k$.
	Define $V_v(X_k):=T_v(X_k)\otimes F_v$. Then we have that $\rk_{F_v} V_v(X_k)\le 2$. Since a division quaternion algebra admits nontrivial representation of degree at least $4$, and $V_v(X_k)$ is a representation for $B_v$, then $\rk V_v(X_k)=0$. By the same argument also  $V_v(Y_k)$ has $F_v$-rank zero. Therefore we obtain
	\[
	V_v(X_k)\simeq V_v(Y_k)\oplus V_v(Y_k)
	\]
	which implies that the desired isogeny exists.
\end{proof}

We say that $x$ is $\mathit{superspecial}$
if the corresponding formal $\calo_{B_v}$-module $X_k$ is isomorphic to $Y_k\oplus Y_k$ where $Y_k$  is as above. This isomorphism is unique up to $\GL_2(\calo_{B_v})$-conjugation.
\\Since $\End(X_k)\simeq M_2(\calo_{B_v})$, the quaternionic action on $X_k$ is given by
\[
\iota\colon\calo_{B_v}\hookrightarrow M_2(\calo_{B_v}).
\]

We also remark that $M_2(\calo_{B_v})\simeq M_2(\calo_{F_v})\otimes_{\calo_{F_v}}\calo_{B_v}$.

\begin{lem}\label{2.7ss}
	We have that
\begin{enumerate}
	\item the set of superspecial points up to isomorphism is in bijection with $B_v^\times$-conjugacy classes of $\iota$;
	\item superspecial points corresponding to the class of a fixed $\iota$ are in bijection with 
	\[
      G'(\Q)_0\backslash G'(\A_f^v)/ U^v
	\]
	
	where $G'(\Q)_0$ are the element in the centralizer of $\iota(\calo_{B_v})$.
	
\end{enumerate}
\end{lem}
\begin{proof}
	See \cite[p.261, Lemma 5.4.5]{sz}.
\end{proof}

\subsection{Connected and Irreducible Components}
Let $Z$ denote the center of $G$, i.e., $Z=\text{Res}_{F/\Q}(F^\times)$.
By strong approximation, the set of connected components $Z_U$ of the Shimura curve defined by (\ref{shi}) is given by 
\[
Z^{an}_U=G(\Q)\backslash\{\pm 1\} \times G(\A_f)/ U\simeq G(\Q)_+\backslash G(\A_f)/ U
\]

and this is isomorphic, via determinant, to
\[
Z(\Q)\backslash Z(\A_f)/\det U.
\]

\begin{lem}\label{2.11}
	The set $\calz_{U,k}$ of irreducible components of the special fiber of $\widehat{\cals}_{k_v}$ is in bijection with
	
	\begin{equation}\label{irredcomp}
		G'(\Q)_e\backslash 
		G'(\A^v_f)/\GL_2(\calo_{K_v})U^v
	\end{equation}

where the subscript ``e" indicates the elements with even order at $v$.
	 
\end{lem}

\begin{proof}
	By the description of the special fiber of $\widehat{\Omega}^\bullet$, it is convenient to separate the two cases. Note that both require an essential use of Cerednik-Drinfeld uniformization as in (\ref{cdu}). Let us first treat the unramified situation.
	
	For the special fibers of $\widehat{\Omega}$, Lemma \ref{dualgraph} show that their irreducible components are in bijection with vertices of $\calt$, i.e., homothety classes of $\calo_{K_v}$-lattices in $K_v^2$. Thus, since the vertices $\calv$ are in bijection with $\PGL_2(K_v)/\PGL_2(\calo_{K_v})$, we have that $\calz_p$ is indexed by
	\[
	G'(\Q)\backslash (\PGL_2(K_v)/\PGL_2(\calo_{K_v})) \times \Z \times G(\A^v_f)/U^v.
	\]
	One can immediately rewrite this double quotient as
	$G'(\Q)\backslash (\GL_2(K_v)/K_v^\times\GL_2(\calo_{K_v})) \times \Z \times G(\A^v_f)/U^v$.
	\\Moreover, since we have $G'(\A_f)\simeq \GL_2(K_v)\cdot G(\A^v_f)$, we obtain that $\calz_v$ is isomorphic to 
	\[
	G'(\Q)\backslash G'(\A_f)/K_v^\times\GL_2(\calo_{K_v})U^v.
	\]
	Finally, to conclude it is enough to note that the order at $v$ of 
	$\det\begin{bmatrix}
		z_v & 
		\\
		 & z_v
	\end{bmatrix}$ 
is even for $z_v\in K_v$.

On the other hand, in the case of $\widehat{\Omega}'$, we need to take into account also the set $\calz_{U,k}'$ parametrizing the exceptional divisors, which arise from the double points, indeed corresponding to a pair of adjacent lattices in $F_v^2$. Since the $\GL_2(F_v)$-action on double points is transitive, we have 
\[
\calz_{U,k}'\simeq \GL_2(F_v)/ S_v
\]

where $S_v$ is the stabilizer of any double point. 
\end{proof} 

\begin{rmk}
	As described in \cite[p.243]{yzz}, the grous $S_v$ is generated by $F_v^\times$, the element 
$\begin{bmatrix}
	& 1
	\\
	\varpi_{F_v}&
\end{bmatrix}$
and the arithmetic subgroup $\Gamma_0(v)$.
\end{rmk}

We conlude by noting that, in terms of linear data, the set of irreducible components is in bijections with (isomorphism classes of) $v$-admissible $(\calo_{B_v},\calo_{B_v})$-bimodules of type $(2,0)$ and $(0,2)$.  Equivalently, in the moduli interpretation, they correspond to those formal modules $X$ which are not special, i.e., those with $\Lie(X)_i$  of $\rk$-$2$ and $\Lie(X)_{i+1}$ of $\rk$-$0$ for $i\in \Z/2\Z$.

\subsubsection{Tower of Local Models}
The main reason to introduce a tower of Shimura curves is to have an adelic action on it.

Let us assume again that the level structures $U$ (and so $U^v$) are sufficiently small.
\\The proper and regular models $\{\cals_U\}_U$ form a projective system with finite flat transition maps, whose inverse limit $\cals$ has a right action of $G(\A_f^v)$ (see \cite[Thm.5.2, p.140]{bc}). 
\\The same construction gives towers of CM points, and special CM points, and of irreducible components, which we denote respectively by $\CM,\SCM$, and $\calz$.

\subsection{Reduction Maps}
We introduce two classes of reduction maps, namely, one from CM points and another one from the connected components.
\\
Let us consider now $k$ to be the residue field modulo a $\mathit{ramified}$ prime. Beside the first vertical injection, the other vertical arrows of the right side of the following diagram is given by natural projections
\begin{equation}\label{natproj}
\xymatrix{
&\scal^{\text{SCM}}_U\ar@{^{(}->}[d]\ar[r]^-{\simeq}\ar@/{}^{-3pc}/[dddd] & T(\Q)_0\backslash G(\A^v_f)/U^v\ar[d]
\\	
&\scal^{\text{CM}}_U \ar[r]^-{\simeq}\ar@{->>}[d] & T(\Q)\backslash G(\A_f)/U\ar[d]
\\
&\cals_{U,k}^{ss}\ar[r]^-{\simeq}\ar@{->>}[d] 
& G'(\Q)_0\backslash G'(\A_f^v)/U^v \ar[d]
\\
&\calz_{U,k} \ar@{.>}[r]^-{\simeq} & G'(\Q)_e\backslash G'(\A_f^v)/\GL_2(\calo_{K_v})U^v &
\\
&\calz_U \ar[r]^-{\simeq}\ar@{->>}[u] & G(\Q)_+\backslash G(\A_f)/U \ar[u]&
}
\end{equation}

where the isomorphisms are proved in Lemma \ref{CM}, Lemma \ref{sCM}, Lemma \ref{2.7ss} and Lemma \ref{2.11} respectively.

A similar diagram holds for the respective towers; we do not rewrite it, since we just notice how it is enough to remove the various level structure and to replace, on the right hand side, the left quotients by their closure. However there is a caveat: this does not work for the irreducible components $\calz_{U,k}$, where it is essential to keep the maximal level at $v$. To the best of our knowledge the construction of good local integral model if considered different level structures is still an open problem.

Denote by $\bar{v}$ a place of $K^{\text{ab}}$ above $v$ with ring of integers $\calo_{\bar{v}}$ and residue field $k_{\bar{v}}$.
\\Since CM points are defined over $K^{\text{ab}}$ and $\cals_U$ is proper over $\Spec \calo_{F_v}$, we have the following maps
\[
\scal_U(K^{\text{ab}})= \cals_U(K^{\text{ab}})\simeq\cals_U(\calo_{\bar{v}})\longrightarrow\cals_U(k_{\bar{v}})
\]
where the bijection follows from the valuative criterion for properness. We thus obtain a $v$-reduction map
\[
\scal^{\text{SCM}}_U\hooklongrightarrow\scal_U^{\text{CM}}\longrightarrow\cals_U(k_{\bar{v}}).
\]

Let $\text{\card}\in\{\text{ss},
 \text{ns}\}$ be a superscript for the reduction of $\cals$, where the first one denotes the superspecial (i.e., singular) locus, and the second one the non-singular locus, that is, $\cals^{ns}$ denotes the irreducible components.
 
To define our reduction maps, we introduce, following \cite[p.38]{cv}, the auxiliary space
\begin{equation}
     \calx^{\text{\card}}_U=\cals^{\text{\card}}_{U,k}\times_{\calz_{U,k}}\calz_U
\end{equation}

so that, by the universal property of the fiber product and the commutativity given by (\ref{natproj}), we obtain the diagram 

\[
	\xymatrix{
	\SCM \ar@/_/[ddr]\ar@/^/[drr]\ar@{.>}[dr]
	\\
	&\cals^{\text{\card}}_{U,k}\times_{\calz_{U,k}}\calz_U\ar[d]\ar[r] & \cals_{U,k}\ar[d]
	\\
	&\calz_U\ar[r] &\calz_{U,k}
	}
\]

whose dotted arrow is unique.
\\
This leads respectively to $\mathit{reduction}$ and $\mathit{irreducible}$ $\mathit{component}$ maps so defined

\begin{align*}
	\SCM&\xrightarrow{red_v} \cals^{\text{\card}}_{U,k}\times_{\calz_{U,k}}\calz_{U}\xrightarrow{pr_2} \calz_U \xrightarrow{c_v} \calz_{U,k}
	\\
	x&\longmapsto
	(red_v(x),c(x)) \longmapsto c(x)\longmapsto
	c_v(x)
\end{align*}

where $pr_2$ is the projection on the second factor.

The resulting composite map $c=pr_2\circ red_v$ commutes with the action of $G_K^{\text{ab}}$ from both sides. 

\subsubsection{Reduction of CM Points.}
In this section we study in detail the reduction of CM points by $v$ both when it ramifies in $K$ and when it remains inert in $K$.


\begin{prop}\label{neron}
	Let $\ccal$ be a generically smooth, geometrically reduced, regular, proper curve of genus $g\ge 1$ over a discrete valuation ring $\calo$, and let $\calc$ be its generic fiber over $E=Frac(\calo)$. Assume moreover that $\ccal$ has semi-stable reduction.
	Then any point $y\in\calc(E)$ reduces to the smooth locus over the residue field of an \'etale extension $\calo'/\calo$.
\end{prop}
\begin{proof}
	Fix a $E$-rational point $b$ on $\calc$. Consider the map 
	\[
	\calc\hooklongrightarrow \Pic(\calc),\;\;y\mapsto [\calo_{\calc}(y-b)]
	\]
	where $\Pic(\calc)$ is defined over $E$. Denote by $\ncal(\Pic(\calc))$ the N\'eron model of $\Pic(\calc)$, defined over $\calo$. We thus obtain the following commutative diagram
	\[
\begin{tikzcd}
	\calc \arrow[hookrightarrow]{r} \arrow[hookrightarrow]{dd}
	& \Pic(\calc) \arrow[hookrightarrow]{d} 
	\\
	& \Pic(\ccal)\ar[d, "\simeq"]
	\\
	\ccal^{\text{sm}} \arrow[hookrightarrow, dashed]{r}
	& \ncal(\Pic(\calc))
\end{tikzcd}
	\]
where the dashed map exists and it is unique (up to isomorphism) by the N\'eron mapping property \cite[p.12]{blr}. Moreover it is a closed embedding by \cite[Thm 1.5]{ros}. 
The isomorphism of $\calo$-schemes is proved in \cite[Thm 4 p.267]{blr}. 
\\Therefore a point $y\in\calc(E)$ extends uniquely to a point $\underline{y}\in\ncal(\Pic(\calc))(\calo)$, i.e., $\underline{y}_E=y$. We have thus obtained a point $\underline{y}$ that generically is in the image of $\ccal^{\text{sm}}$. 
\\By \cite[Prop.2, p.13]{blr} we also have that the map
\[
\ncal(\Pic(\calc))\otimes_\calo\calo'\longrightarrow\ncal(\Pic(\calc)\otimes_\calo\calo')
\]
is an isomorphism  only if  $\calo'/\calo$ is \'etale.  Since the reduction to the special fiber is obtained by scalar restriction
we conclude.
\end{proof}

\begin{rmk}
Here we collect few technical comments on Proposition \ref{neron}.
	\begin{enumerate}
		\item It is not necessary to assume the existence of the $E$-rational point $b$. If such a point does not exists then the result remains (vacously) true taking $y$ as a base point;
		
		\item the hypothesis that $\ccal$ has semi-stable reduction is not restrictive. By the (strong form of the) semistable reduction theorem \cite[Thm 3.5]{con2} $\ccal$ has semi-stable reduction upon passage to an \'etale extension $\calo'/\calo$. In fact one may pass to any cover over which all of the $p$-torsion of $\Pic(\calc)$ is defined, where $p$ is an odd prime not equal to the residue characteristic of $\calo$. 
	\end{enumerate}
\end{rmk}

\begin{rmk}
	Proposition \ref{neron} hopefully explains the geometric meaning of the sentence in \cite[pp.242]{yzz} about the reduction of CM points at an inert prime.
\end{rmk}

In our context, Proposition \ref{neron} was motivated by Proposition \ref{vram}. Nonetheless, we also need an ad hoc proof based on the modular interpretation of our points.

\begin{lem}\label{2.17}
	Let $v$ a place of $F$ which ramifies in $K$. 
	
	\begin{enumerate}
	\item The $(\calo_{B_v},\calo_{B_v})$-bimodule associated to the $v$-reduction of a CM point is $v$-admissible of type $(1,1)$;
	
	\item  $v$-admissible $(\calo_{B_v},\calo_{B_v})$-bimodules of type $(1,1)$ and $\rk$-$8$ correspond to special formal $\calo_{B_v}$-modules of height $4$ which are superspecial over $k$, up to respective isomorphisms.
	\end{enumerate}
\end{lem}
\begin{proof}
	 By \cite[p.20]{rib} this bimodule is isomorphic to $\calo_{B_v}\otimes_{\calo_{K_v}}\calo_{B_v}$.  Note that its rank is $8$ over $\calo_{F_v}$ since $\calo_{B_v}$ is locally free of $\rk$-$2$ over $\calo_{K_v}$. The proof that it is $v$-admissible of type $(1,1)$ follows almost verbatim the proof of \cite[Thm.2.4]{rib}.
	 
	 As showed in Section \ref{lRb} by the bijection (\ref{rbim}), to every $(\calo_{B_v},\calo_{B_v})$-bimodule of $\rk$-$8$ we can associate a morphism $f\colon\calo_{B_v}\hooklongrightarrow M_2(\calo_{B_v})$, which corresponds to the quaternionic action on a superspecial point. By \cite[p.39]{rib}, a bimodule of type $(1,1)$ corresponds to a special\footnote{We caution the reader that Ribet used the term $\mathit{mixed}$ in \cite{rib}.} formal module.
	 \end{proof}

\begin{prop}\label{vram}
	 The reduction of  CM points in the special fiber modulo $v$ corresponds to superspecial points if and only if $v$ ramifies in $K$. In symbols,
	\[
	red_v(\CM)\subseteq
	\begin{cases} 
		&\cals^{\text{ss}}_{k},\; \text{for}\;\; $v$\;\; \text{ramified in}\;\;$K$\\ 
		&\cals^{\text{ns}}_{k},\;\;\text{for}\;\;$v$\;\;\text{inert in}\;\;$K$.
	\end{cases}
	\]
\end{prop}
\begin{proof}
	Suppose that $v$ ramifies in $K$. 
	By what discussed in Section \ref{lRb}, to every $v$-reduction of a CM point we can associate a $(\calo_{B_v},\calo_{B_v})$-bimodule of $\rk$-$8$. By Lemma \ref{2.17}, this bimodule 
	is $v$-admissible of type $(1,1)$. 
	
	Again by Lemma \ref{2.17} and by the first bijection of (\ref{Dieudonne})  we obtain the following bijections
	\[
	\Set{\begin{array}{c}
			\text{$v$-admissible}\\
			\text{$(\calo_{B_v},\calo_{B_v})$-bimodules}\\
			\text{of type $(1,1)$, $\rk$-$8$}\\
	\end{array}}
\longleftrightarrow
	\Set{\begin{array}{c}
			\text{Special formal}\\
			\text{$\calo_{B_v}$-modules of height $4$}  \\
			\text{superspecial over $k$}\\
	\end{array}}
	\longleftrightarrow
\text{Sing}(\widehat{\cals}_k)
   \]
	where the first two sets are considered up to their respective isomorphisms.
	
	On the other hand, the case of $v$  inert  follows from Proposition \ref{neron}, by taking $\ccal=\widehat{\cals}_U$ and $y$ a CM point. Nonetheless, we give an alternative argument via the theory of Ribet's bimodules.
	
	Consider  a special CM point $x$ which reduces to a superspecial point. Denote by $X$ and $X_k$ their formal module and the its reduction  respectively, where $k$ is the residue field at $v$. 
	
	We have therefore the following embeddings
	\[
	\End(X)\simeq \calo_{K_v}\hooklongrightarrow \End(X_k)\simeq M_2(\calo_{F_v})\otimes \calo_{B_v}, 
	\]
	where the first isomorphism is given by $x$ being a special CM point.
	\\Again by the above moduli correspondences, we have that a superspecial point corresponds to a $v$-admissible bimodule of type $(1,1)$, which corresponds to a morphism $f\colon \calo_{B_v}\hookrightarrow M_2(\calo_{B_v})$. By Lemma \ref{eichorder} we have that the commutant of $f(\calo_{B_v})$ in $M_2(\calo_{B_v})$ is an Eichler order of level $\varpi_{F_v}$. We recall that CM points $x$ by $K_v$ there is an $\End(x)=\calo_{K_v}$-action commuting with the $\calo_{B_v}$-action. It follows from \cite[Thm.3.2]{vign} that $\calo_{K_v}$ embeds in an Eichler order of level $\varpi_{F_v}$ if and only if $v$ ramifies in $K$.

	We therefore conclude that all special CM points have superspecial reduction at a prime $v$ which ramifies in $K$.
\end{proof}

\subsubsection{An Auxiliary Space}\label{auxspsec}

Let $\nr\colon G'(\A_f)\rightarrow G'(\A_f^v)\times F_v^\times$ denote the morphism induced by the reduced norm on $B^{',\times}_v$ and by the identity elsewhere.

\begin{lem}
	The tower $\calx^{\text{ss}}$ is in bijection
	with
	\[
	\overline{\nr(G'(\Q)_0)}\backslash G'(\A_f^v)\times F_v^\times
	\]
\end{lem}
\begin{proof}
	Since, by definition, $\calx^{\text{ss}}=\cals_{U,k}^{\text{ss}}\times \calz_U$, we thus have that its points are in bijection with
	\[
	(G'(\Q)_0\backslash G'(\A_f^v)/ U^v)\times(G(\Q)_+\backslash F_v^\times \times G(\A_f^v)/U)
	\]
	 by Lemma \ref{2.7ss}.
	\\As in  \cite[p.259]{sz}), we can write
	$G'(\Q)_0\backslash G(\A^v_f)/U^v \simeq G'(Q)\backslash G'(\A_f)/\calo_{B_v}^\times U^v$
	and then we conclude by  \cite[Cor. 3.10, p.42]{cv}.
\end{proof}

We finally need to introduce one more auxiliary object, which we denote by 
\[
\cala_v:=\overline{G'(\Q)}_0\backslash  G'(\A_f)
\]
where we point out that  $G'(\A_f)\simeq G'(\A_f^v)\times G'(F_v)$.

We have thus the following commutative triangle
\[
\xymatrix{
&\cala_v \ar[d]^-{q} \ar@{.>}[rd]
\\
&\calx^{\text{ss}} \ar[r]^-{c_v} & \calz_k
}
\]
where the morphism $q$ is induced by $\nr$.

\subsubsection{Simultaneous Reductions}\label{simred}
Let $\Sigma$ be a finite set of non-archimedean places of $F$, and take $\Sigma$  disjoint from $S$. Suppose that $B$ ramifies at every place of $\Sigma$.
Let $R$ be a finite subset of $G_K^{ab}$. Then we consider
\[
\calx^{\Sigma\times R}:=\prod_{v\in \Sigma,R}\calx_{U,k_v}
\]
where $k_v$ is the residue field modulo $v$. With the same notation we also write $\calz^{\Sigma\times R}$.

There is a natural generalization of our two main maps previously defined; namely, we have the $\mathit{simultaneous}$ reduction map
\begin{align*}
Red\colon\SCM&\longrightarrow\calx^{\Sigma\times R}
\\
x&\longmapsto (red_v(\sigma x))_{v\in \Sigma, \sigma\in R}
\end{align*}
and the $\mathit{simultaneous}$ irreducible component map
\begin{align*}
	C\colon \calx^{\Sigma\times R}&\longrightarrow \calz^{\Sigma\times R}
	\\
	(x_{v,\sigma})_{v,\sigma}&\longmapsto(c_v(x_{v,\sigma}))_{v,\sigma}.
\end{align*}


\section{Equidistribution}\label{equisec}

Here we state the main equidistribution result, whose proof is completed in Section \ref{proof}. Note that we  write it in the simultaneaous case, although the various steps of the proof are (mostly) in the case of a single prime, for the sake of notational simplicity.
\\

Let $\ell$ denote a maximal ideal of $\calo_F$ such that $\ell\notin S\cup\Sigma$, so that $B_\ell\simeq M_2(F_\ell)$. Moreover, we do not impose any restriction on $\ell$ relatively to $K$.

Let us here introduce two important technical definitions. 

A $\ell$-$\mathit{isogeny}$ $\mathit{class}$ $\calh$ in $\SCM$ is a $B_\ell^\times$-orbit in $\SCM$. Two points $x,y\in S^{an}_U$ are $\ell$-$\mathit{isogenous}$ if the corresponding $g,g'\in G'(\A^v_f)$ coincide on every non-archimedean place of $F$ different from $\ell$.

\begin{rmk}
	In the case of $B=M_2(\Q)$, i.e., $\scal$ is a modular curve, two CM points are $\ell$-isogenous if, for some $n$, there exists a degree-$\ell^n$ isogeny between them.
\end{rmk}


We write, for $f\colon\CM\rightarrow\C$ and $a\in\C$, the limit $\lim_{x\to\infty}f(x)=a$ if for all $\epsilon>0$ there exists a compact $\ C_\epsilon$  in $\CM$ such that
$|f(x)-a|<\epsilon$ for all $x\in\calh-C_\epsilon$.

Let $\Rat_K(\ell)$ be the subgroup of $G_K^{ab}$ of $\ell$-$\mathit{rational}$ elements, that is, those $\sigma\in G_K^{ab}$ such that $\sigma=\rec_K(\lambda)$ for $\lambda\in \widehat{K}^\times$, whose $\ell$-adic component $\lambda_\ell$ belongs to $K^\times\cdot F_\ell^\times\le K^\times_\ell$.

\begin{thm}\label{mainthm}
	Let $\calg$ be an open compact subgroup of $G_K^{ab}$ with Haar measure $dg$, and let also $R\subset G_K^{ab}$ be a finite subset of element pairwise disjoint modulo $\Rat_K(\ell)$. Then 
	\begin{equation}\label{equi}
		\lim_{x\to\infty}\frac{1}{\vol(\calg)}\int_{\calg}f\circ Red(gx)dg=\int_{C^{-1}(g\bar{x})}f d\mu
	\end{equation}
	
	for every $f\in\calc(\calx^{\Sigma\times R},\C)$, where $\vol(\calg):=\int_{\calg}dg$ and $\mu$ is the measure on the fiber $C^{-1}(g\bar{x})$ constructed in Section \ref{mof}. 
	
	Moreover, let $\overline{\calh}$ be the image of $\calh$ in $\scal_U^{\text{SCM}}$, and consider the characteristic functions of the elements of $\calx^{\Sigma\times R}_U$. Then, for almost all $x\in\overline{\calh}$, we obtain
	\[
	Red(\calg\cdot x)=C^{-1}(\calg\cdot (C\circ Red(x)))
	\]
	in $\calx_U^{\Sigma\times R}$.
\end{thm}

\subsubsection{Measures on Fibers}\label{mof}
Recall from Section \ref{auxspsec} that by $\nr$ we denote the map (essentially) induced by the reduced norm on $B_v^{',\times}$, namely, $\nr\colon G'(\A_f)\rightarrow G'(\A_f^v)\times F_v^\times$. Let $G'(\A_f)^1:=\ker(\nr)$. In this section we follow the lines of \cite[Section 2.4.1]{cv} to prove the existence of family of measures on the fibers of $q\circ c$ and $c$.

\begin{lem}\label{fibstab}
Let $z$ be in $\calz$.  We have that:
	\begin{enumerate}
		\item the fibers $(c\circ q)^{-1}(z)$ are $G'(\A_f)^1$-orbits in $\cala_v$;
		\item the fibers $c^{-1}(z)$ are $(G'(\A_f^v)\times F_v^\times)^1$-orbits in $\calx^{\text{ss}}$;
		\item for $x=\overline{G'(\Q)_0}g\in\cala_v$, we have that 
		\[
		\Stab_{G'(\A_f)^{1}}(x)=g^{-1}G_0'(\Q)g
		\]
		is discrete and cocompact in $G'(\A_f)^1$.
		\item for $x=\overline{\nr(G'(\Q)_0)}g\in\calx^{\text{ss}}$, we have
		\[
		\Stab_{G'(\A_f^v\times F_v^\times)^{1}}(x)=g^{-1}\nr(G_0'(\Q))g
		\]
		is discrete and cocompact in $(G'(\A_f^v)\times F_v^\times)^1$.
	\end{enumerate}
\end{lem}
\begin{proof}
	The proof essentially follows from \cite[Lemma 2.16,p.18]{cv}. In our setting we work for a single place $v$,  with $G'(\Q)_0$ instead that $G(\Q)$. 
	
	For $(2)$ and $(4)$, we note that since we have  the isomorphism of topological groups $\cala_v/\ker(\nr_v)\simeq \calx^{\textit{ss}}$ and that $G'(\A_f)\simeq G'(\A_f^v)\times F_v^\times \times \ker(\nr_v)$.
\end{proof}

\begin{cor}\label{unif}
	For $z\in\calz$  the following homeomorphisms hold.
	    \begin{enumerate}
		\item For $x\in(c\circ q)^{-1}(z)$, we have $Stab_{G'(\A_f)^{1}}(x)\backslash G'(\A_f)^1 \simeq (c\circ q)^{-1}(z)$;
		\item for $x\in c^{-1}(z)$, we have $\Stab(x)\backslash (G'(\A_f^v)\times F_v^\times)^1\simeq c^{-1}(z)$.
		\end{enumerate}
	\end{cor}
\begin{proof}
	By Lemma \ref{fibstab}, the $G(\A_f)^1$-equivariant homeomorphism is given by  $g\mapsto gx$.
\end{proof}

We conclude by noting that Lemma \ref{fibstab} and Corollary \ref{unif} imply the existence and uniqueness of the probability measures

\begin{itemize}
	\item $\mu_z$ on $(c\circ q)^{-1}(z)$
	\item $\mu$  on $c^{-1}(z)$
\end{itemize}

which are $G'(\A_f)^1$-invariant and $G'(\A_f)\times F^\times_v$-invariant respectively.
\\

Let $\mu^1$ a Haar measure on  $G'(\A_f)^1$. Here we explain how to describe an induced family of Haar measure $(\mu^1_z)_z$ on the fibers $(c\circ q)^{-1}(z)$ for $z\in\calz$. This construction that can be viewed as an analogue of integration along fibers.

For  a compact open subgroup $U^1$ of $G'(\A_f)^1$,  let $pr\colon G'(\A_f)^1\rightarrow G'(\A_f)^1/U^1$ be the canonical projection  and consider a point $x$ of the fiber $(c\circ q)^{-1}(z)$. Let $\bar{\mu}$ be the Haar measure of $G'(\A_f)^1/U^1$ given by the pushforward of $\mu^1$ by $pr$.
\\By \cite[Prop.4.a, p.31]{bour} we have
\[
\int_{(c\circ q)^{-1}(z)}\mathbf{1}_{G'(\A_f)^1/U^1}\;d\bar{\mu} \cdot\int_{(c\circ q)^{-1}(z)}\mathbf{1}_{xU^1}\;d\mu^1_z=\int_{G'(\A_f)^1}\mathbf{1}_{U^1}\;d\mu^1
\]

where the second integral is called the $\mathit{orbit}$ $\mathit{means}$ of $\mathbf{1}_{xU^1}$.
\\Since for any function $f$ on $G'(\A_f)^1/U^1$ the composition $f\circ pr$ is constant on the orbits, we immediately obtain
\[
\#\Stab_{U^1}(x)=\int_{(c\circ q)^{-1}(z)}\mathbf{1}_{G'(\A_f)^1/U^1}\;d\bar{\mu}.
\]

This characterizes uniquely the measures $\mu^1_z$'s by the following formula
\begin{equation}
	\mu^1_z(xU^1)=\dfrac{1}{\#\Stab_{U^1}(x)}\mu^1(U^1)\;.
\end{equation}
which Orbit-Stabilizer theorem can be seen as a very special case of. Note that the finiteness of $\Stab_{U^1}(x)=\Stab_{G^1_S(\A_f)}(x)\cap U^1$ comes from its description as an intersection of  a discrete and a compact set.

\subsubsection{Quaternionic Uniformization}
Fix a place $\ell$ such that $\ell\notin \Sigma$ and  $B$ is split at $\ell$. We thus denote by $B^1_{\ell}:=\{b\in B_{\ell}: \nr(b)=1 \}$ the reduced norm-$1$ quaternions. Indeed we have $B^1_{\ell}\simeq\SL_2(F_\ell)$.

\begin{lem}\label{quatunif}
	For $x\in c^{-1}(z)/H$, the map
	\[
	\Stab(x)\backslash B_\ell^1\simeq c^{-1}(z)/U
	\]
	
	is a $B_\ell^1$-equivariant homeomorphism.
\end{lem}
\begin{proof}
	Let $U$ be a compact open subgroup of $G'(\A_f^\ell)^1$. The $G'(\A_f^\ell)^1$-action on the fibers induces a $B_\ell^1$-action on $(c\circ q)^{-1}(z)$ for $z\in\calz$. We obtain that this action is transitive, and the stabilizer of $x$ is discrete and cocompact in $B^1_\ell$. This follows by adapting \cite[Lemma 2.19]{cv} to our spaces.
	
	Let $x=\overline{G'(\Q)}_0gU$, with $g\in G'(\A_f)$. By Strong Approximation (see \cite[Thm 28.2.10, p.481]{voi}), we have, for $\tilde{x}\in\overline{G'(\Q)}_0g_1$, with $g_1\in G'(\A_f)^1$,
	\[
	\Stab(\tilde{x})B_\ell^1 U=G'(\A_f)^1,
	\]
	and by Lemma \ref{fibstab} we have $\Stab(\tilde{x})=g_1^{-1}G'(\Q)_0^1 g_1$.  It immediately follows that 
	\begin{equation}\label{lemma2.20}
	(c\circ q)^{-1}(z)=\tilde{x}G'(\A_f)^1=\tilde{x}B_\ell^1
	\end{equation}
	and therefore $B_\ell^1$ acts transitively on $(c\circ q)^{-1}(z)/U$.
	
	Now let $U'$ be a compact open subgroup of $(G'(\A_f^v)^1\times F_v^\times)^\ell$. Then again there is a $B_\ell^1$-action on $c^{-1}(z)/U'$, and as before, for $x\in c^{-1}(z)/U'$, this action is transitive and $\Stab(x)$ is discrete and cocompact in $B_\ell^1$. For $x=\overline{\nr(G'(\Q)_0)}gU'$ with $g\in G'(\A_f^v)\times F_v^\times$, the stabilizer of $x$ can be determined explicitly as in \cite[Lemma 2.20]{cv}. Let $U=\nr^{-1}(U')$ a compact open in $G'(\A_f)^1$. We have that $q$ induces a $B_\ell^1$-equivariant homeomorphism 
	\[
	(c\circ q)^{-1}(z)/U'\simeq c^{-1}(z)/U
    \]
    and  equation (\ref{lemma2.20}) yields that
    \[
    \Stab(x)\backslash B_\ell^1\simeq c^{-1}(z)/U
    \] 
    is a $B_\ell^1$-equivariant homeomorphism.
\end{proof}

\begin{rmk}
 We underline how the hypothesis on $\ell$ is essential. In particular, if $\ell\in\Ram(B)$, then $B'_\ell$ would be compact and thus Strong Approximation would not apply.
\end{rmk}

\subsubsection{Iwasawa-like Decomposition}
This decomposition is a crucial technical step in the proof of Theorem \ref{mainthm}, as it shows how to move from Galois orbits to unipotent orbits, which is essential in order to apply Ratner's theorem.
\\

Let $u\colon F_\ell\rightarrow G$ be a continuous, additive morphism such that \newline$du(t)/dt|_{t=0}\neq 0$. Then we call its image $P^t=\{u(t): t\in F_\ell\}$  a $\mathit{one}$-$\mathit{parameter}$ subgroup.  We point out that for $\GL_n(F_\ell)$, every one-parameter subgroup consists of unipotent elements.
\\For a nilpotent $N\in B_\ell^\times$, we consider the one-parameter $\mathit{unipotent}$ subgroup defined by $u(t)=1+tN$.

\begin{lem}\label{iwas}
	Let $\calh$ be a $\ell$-isogeny class.
	There is a finite set of indexes $I$, and for each index $i$ there is a $x_i\in\calh$ and a one parameter unipotent subgroup $P^t_i$ such that
	\[
	\calg\cdot \calh \cdot U=\bigcup_{i\in I}\bigcup_{n\ge 0}\calg\cdot x_i u_{i,n}\cdot U 
	\]
	where $u_{i,n}=u_i(\varpi_\ell^{-n})\in u_i(D)$ for $D$ some compact open subgroup of $F_\ell^\times$.
\end{lem}
\begin{proof}
   Our setting hardly changes the proof given in \cite[Section 2.6, p.29]{cv}. One of the few details affected by having a $\ell$-isogeny class of $\mathit{special}$ CM points is that, by Lemma \ref{SCM}, any element $x_0\in\calh$ is such that $x_0=\overline{T(\Q)}_0g_0$ with $g_0\in G'(\A_f^p)$.
\end{proof}


\subsubsection{Ratner's Theorem}
We refer for the main reference of these crucial results to the seminal Ratner's paper \cite{rat}. 

Let us begin by recalling some terminology so to restate Ratner's theorems.
\\Let $G$ be a $\ell$-adic Lie group, and let $\Gamma$ be a discrete and cocompact subgroup of its\footnote{More generally, one requires that $\Gamma\backslash G$ admits a finite invariant measure.}. We say that a subset $A\subset \Gamma\backslash G$ is $\mathit{homogeneous}$ if there exists $x\in\Gamma\backslash G$ and a closed subgroup $H\le G$ such that $xHx^{-1}\cap \Gamma$ is a lattice in $xHx^{-1}$ and $A=xH$ (so that $A$ is closed).

We state a less general version of \cite[Thms. 3,6]{rat}, which still suffices for our goals.

\begin{thm}\label{ratner}
	Let $G$, $\Gamma$, $H$, and $U\le H$ as above.
	\begin{enumerate}
		\item For any $x\in \Gamma\backslash G$, then the closure $\overline{xU}$ is homogeneous;
		\item for $f\in\calc(\overline{xU},\C)$, we have
		\begin{equation}
			\lim_{s\to\infty}\dfrac{1}{\vol(D_s)}\int_{D_s}f(xu(t))d\lambda(t)=\int_{\overline{xU}}f(y)d\mu(y)
		\end{equation}
	for the compact open $D_s=\{a\in F_\ell : |a|\le s \}$  and $\mu$ the unique $H$-invariant measure supported on $\overline{xU}$.
	\end{enumerate}
\end{thm}

We remark that part $(1)$ shows that the closure of any such orbit is a nice geometric subset of $\Gamma\backslash G$, while part $(2)$ can be viewed an equidistribution result; namely, the orbit $xU$ is equidistributed in its closure.

In our context, we apply these powerful ideas to $G=(B_\ell^1)^r=\SL_2(F_\ell)^r$ for a positive natural $r$, and $\Gamma=\prod_{i\le r}\Gamma_i$ for $\Gamma_i$ lattices in $\SL_2(F_\ell)$.

\begin{lem}\label{tec}
	Suppose that $Red(xU)$ is dense in $C^{-1}(\bar{x})$.
	Moreover for $f\in\calc(C^{-1}(x),\C)$ we have
	\[
	\lim_{s\to\infty} \dfrac{1}{\vol(D_s)}\int_{D_s}f\circ Red(xu(t))d\lambda(t)=\int_{C^{-1}(\bar{x})}fd\mu_{\bar{x}}.
	\]
\end{lem}
\begin{proof}
		
	Assume, without loss of generality, that $f$ is locally constant, so that it factors for an appropriate open compact $U$ through $C^{-1}(x)/U$. 
	
	Let recall that $\Sigma$ denote a finite set of non-archimedean places of $F$ disjoint from $S$ while $R$ denote a finite subset of $G_K^{\text{ab}}$. By Lemma \ref{lemma2.20} we have the following $B_\ell^1$-equivariant homeomorphism
	\[
	\Gamma_x\backslash (B_\ell^1)^{\Sigma\times R}\simeq C^{-1}(\bar{x})/U
	\] 
	where $\Gamma_x$ is the stabilizer of $Red(x)U$ in $(B_\ell^1)^{\Sigma\times R}$. 
	
	Let $\Delta\colon B_\ell^1\hookrightarrow (B_\ell^1)^{\Sigma\times R}$ be the diagonal embedding. Under this homeomorphism the image of $(t\mapsto Red(x.u(t)))$ in $C^{-1}(\bar{x})/U$ corresponds to the image of $(t\mapsto \Delta\circ u(t))$ in $\Gamma_x\backslash (B_\ell^1)^{\Sigma\times R}$. Similarly, there is a correspondence of probability measure on the two spaces.
	
	By the second part  of Ratner's theorem \ref{ratner} it immediately follows that
	\[
	\lim_{s\to\infty}\dfrac{1}{\vol(D_s)}\int_{D_s}f(\Delta\circ u(t))d\lambda(t)=\int_{\Gamma_x\backslash(B_\ell^1)^{\Sigma\times R}}f d\mu
	\] 
	and by the above correspondence we conclude.
\end{proof}

\begin{lem}
	Under the assumptions of Theorem \ref{mainthm}, $Red(\gamma xU)$ is dense in $C^{-1}(\gamma \bar{x})$ for almost all $\gamma\in 
	G_K^\text{ab}$.
\end{lem}
\begin{proof}
The proof follows almost verbatim \cite[Section 2.5.2]{cv}.
\end{proof}

\subsubsection{Proof of Theorem \ref{mainthm}}\label{proof}

Armed with the results of the previous section, we can finally prove formula (\ref{equi}). In order to do so, we need some contributory functionals as follows. We define, for $x\in\SCM$ and $z\in\calz$,
\begin{itemize}
	\item $\delta(-,x)\in\calc(\calx^{\text{\card}},\C)^*$ as $\delta(f,x):=\int_{\calg}f\circ Red(gx) dg$;
	\item $I(-,z)\in\calc(\calz,\C)^*$ as $I(f,z):=\int_{c^{-1}(z)}fd\mu_z$;
	\item $B(-,x)\in\calc(\calx^{\text{\card}},\C)^*$ as $\int_\calg I(f,g\bar{x})dg$.
\end{itemize}
where $\bar{x}=C\circ Red (x)$.

Assume, without loss of generality, that the function $f\in\calc(\calx^{\text{\card}},\C)$ is locally constant. Let $U$ a compact subgroup of $G'(\A_f)$ such that $f$ factor through $\calx^{\text{\card}}_U$. This implies that also the functionals $\delta(f,x)$, $B(f,x)$ and $I(f,z)$ factor through $\calg\backslash \scal^{\text{SCM}}_U$ and $\calz_U$ respectively.

By Lemma \ref{tec} we have
\begin{equation}\label{r}
\lim_{s\to\infty}\dfrac{1}{\vol(D_s)}\int_{D_s}f\circ Red(g.x.u(t))d\lambda(t)=\int_{C^{-1}(g.\bar{x})}fd\mu.
\end{equation}
Integrating both sides of (\ref{r}) over $\calg$ and applying Lebesgue dominated convergence  to move the limit outside the integral over $\calg$ and Fubini theorem to switch the order of the integration, we obtain
\begin{equation}\label{s}
\lim_{s\to\infty}\dfrac{1}{\vol(D_s)}\int_{D_s}\delta(x.u(t))d\lambda(t)=B(f,x)
\end{equation}
for every $x$ and $u$. From Lemma \ref{iwas}, we have that, for $i\in I$, the map $t\mapsto \delta(x_i.u_i(t))$ is constant on $D_s$, namely it is $\delta(x_iu_{i,s})$. By (\ref{s}), for all $i\in I$ we have 
\[
\lim_{s\to\infty}\delta(x_iu_{i,s})=B(x_i).
\] 
Pick $\epsilon>0$ and $N$ big enough such that for all $i$ and for $s>N$ we have $|\delta(x_iu_{i,s})-B(x_i)|<\epsilon$.
\\Define $C_\epsilon$ to be the following compact subset of $\scal^{\text{SCM}}$
\[
C_\epsilon=\bigcup_{i\in I}\bigcup_{n\le N}\calg\cdot x_iu_{i,s}(D_s)\cdot U.
\]
Since for $x\in \calh$ there exists an $i\in I$ and a positive $s$ such that $x\in \calg\cdot x_iu_{i,s}\cdot U$, then $\delta(f,x)=\delta(x_iu_{i,s})$ and $B(f,x)=B(x_i)$. For $x$ outside the compact $C_\epsilon$, we obtain 
\[
|\delta(f,x)-B(f,x)|<\epsilon
\]
and so this completes the proof of Theorem \ref{mainthm}.

\subsubsection{Equidistribution at a Single Prime}

Consider the case of a single prime $v$ ramified in $K$.
Geometrically, we recall that superspecial points correspond to the intersection points in the  the special fiber of $\widehat{\cals}$.
Since they form a finite set, the equidistribution expressed by (\ref{equi}) becomes for a fixed $s\in\cals^{\textit{ss}}$
\begin{equation}\label{equiram}
	\lim_{x\to\infty}\dfrac{1}{\# \calg x}\sum _{\substack{g\in\calg \\ red_v(gx)=s}}\mathbf{1}_s(g)= O_\calg^{-1} I(s)
\end{equation}

where $O_\calg$ is the sum of the the cardinalities of all $\calg$-orbits in the connected components, and $I(s)$ is the positive constant as in \cite[Cor.2.11]{cv}.

Indeed the functionals introduced in Section \ref{equisec} thus become
\begin{itemize}
	\item $\delta(x):=\delta(\mathbf{1}_s,x)=\sum_g\mathbf{1}_s(g)$, for  $g\in\calg$ such that $red_v(gx)=s$;
	\item $I(s):=I(\mathbf{1}_s,s)$ the value of $I$ on $C(s)$; 
	\item $B(x):=B(\mathbf{1}_s,x)=\begin{cases}
		& 0,\;\;\text{if}\;\;\bar{x}\notin\calg\cdot C(s)
		\\
		& O_\calg^{-1} I(s),\;\;\text{otherwise}.
	\end{cases}$
\end{itemize}

The case of $v$ inert in $K$ follows from a slight change of notation.

\subsubsection{Equidistribution on a Product of Shimura Curves}
Given our main equidistribution theorem, we show how for a product of Shimura curves the analoguous result holds. For the sake of notational simplicity, we deal only with the product of two Shimura curve.
\\

Let $X$ be a locally compact metric Hausdorff space, and consider $\text{Meas}(X):=\Hom_\C(\calc(X,\C),\C)$, namely, the space of regular Borel measures on $X$ which we endow with  the weak-$*$ topology.

\begin{lem}\label{tensor}
 Let
	\[
	\otimes\colon\text{Meas}(X)\times\text{Meas}(X)\longrightarrow\text{Meas}(X\times X)
	\]
denote the tensor product of measures and let $\Delta\colon\text{Meas}(X)\hookrightarrow\text{Meas}(X)\times\text{Meas}(X)$ denote the diagonal embedding. Then $\otimes\circ\Delta$ is a continuous map. 
\end{lem}

\begin{proof}

	Suppose that $\{\mu_n\}_n$ is a net in  $\text{Meas}(X)$, converging in the weak-$*$ topology to $\mu$. We recall that the weak-$*$ topology on the space of measures on $X$ is the coarsest topology for which tha maps $\mu\mapsto \mu(f)$ are continuous for $f$ continuous. For such an $f$, this means that we have 
	\[
	\mu(f)=\lim_n\mu_n(f)
	.
	\]
Consider a subspace of the space of measures on $X$ of the form $U\times V$ for $U$ and $V$ measurable. If $X$ is profinite, and thus non-archimedean, then we require that $U$ and $V$ are compact open, so that the two functions $\mathbf{1}_U$ and $\mathbf{1}_V$ are continuous. On the other hand, for $X$ archimedean, $\mathbf{1}_U$ and $\mathbf{1}_V$ can be substituted by continuous functions by the mean of Urysohn Lemma\footnote{In our case, it essentially reduces to the Pasting Lemma.}. Then we have
	\begin{equation}
		\begin{split}
			(\mu\otimes \mu)(\mathbf{1}_{U\times V})&=\mu(\mathbf{1}_{U})\mu(\mathbf{1}_{V})
			\\
			&=\lim_n\mu_n(\mathbf{1}_U)\lim_n\mu_n(\mathbf{1}_V)
			\\
			&=\lim_n(\mu_n\otimes\mu_n)(\mathbf{1}_{U\times V}),
		\end{split}
	\end{equation}
where the first and the third equalities follows by Fubini theorem. Note that the first equality we also use that $\mathbf{1}_{U\times V}=\mathbf{1}_U\cdot\mathbf{1}_V$. Now, any function $f$ on $X\times X$ is obtained as limit of linear combinations of functions of the form $\mathbf{1}_{U\times V}$, i.e., for complex $c_i$'s,
\[
f=\lim_{m\to\infty}\sum_{i=0}^{m}c_i\mathbf{1}_{U_i\times U_i}.
\]
Thus we have
\begin{equation}
	\begin{split}
		(\mu\otimes\mu)(f)&=(\mu\otimes\mu)(\lim_{m\to\infty}\sum_{i=0}^{m}c_i\mathbf{1}_{U_i\times U_i})
		\\
		&=\lim_{m\to\infty}\sum_{i=0}^m(\mu\otimes\mu)(c_i\mathbf{1}_{U_i\times U_i})
		\\
		&=\lim_{m\to\infty}\sum_{i=0}^m(\lim_n (\mu_n\otimes\mu_n)(c_i\mathbf{1}_{U_i\times U_i}))
	\end{split}
\end{equation}
where the second equality follows by Levi theorem. Finally, we can switch the order of the limits of the last equality by the Monotone Convergence theorem, so that we obtain 
\[
(\mu\otimes\mu)(f)=\lim_n(\mu_n\otimes\mu_n)(f)
\]
which implies the desired weak-$*$ convergence.
\end{proof}

Let $\calh\times\calh$ be the product of $\ell$-isogeny classes in the fiber product $\scal^\text{SCM}\times\scal^\text{SCM}$. Consider also $F=(f,f)\in\calc((\calx\times\calx)^{\Sigma\times R},\C)$ where $f$, $\Sigma$ and $R$ are defined as in Theorem \ref{mainthm}. Moreover we also introduce the maps
\[
\underline{Red}\colon\scal^\text{SCM}\times\scal^\text{SCM}\longrightarrow(\calx\times\calx)^{\Sigma\times R}
\]
and
\[
\underline{C}\colon(\calx\times\calx)^{\Sigma\times R}\longrightarrow (\calz\times\calz)^{\Sigma\times R},
\]
where $\underline{Red}:=(Red,Red)$ and $\underline{C}:=(C,C)$ for $Red$ and $C$  defined as in Section \ref{simred}.

\begin{prop}
	Let $\underline{x}=(x,y)$ for $x,y\in\overline{\calh}$ and let $\calg$ be a compact open in $G_K^\text{ab}$.
	Then we have
	\[
	\lim_{\underline{x}\to\infty}\dfrac{1}{\vol(\calg\times\calg)}\int_{\calg\times\calg}F\circ \underline{Red}(\underline{g}\underline{x})d\underline{g}=\int_{\underline{C}^{-1}(\bar{\underline{x}})}Fd\underline{\mu}
	\]
	where $\vol(\calg\times\calg)=\int_{\calg\times\calg}d\underline{g}$ for $d\underline{g}$ the Haar measure induced by $dg$ and $d\underline{\mu}$ is induced by the measure $\mu$ on $C^{-1}(g\bar{x})$.
\end{prop}
\begin{proof}
	By Theorem \ref{mainthm} we have the convergence  for the measures $\delta(-,x)=\int_\calg (-)\circ Red(gx)dg$. Applying Lemma \ref{tensor} with $\mu_n=\delta(-,x)$ implies the desired result.
\end{proof}

\section{Appendix: the definite case}

Let $R$ be an Eichler order in $B$, a definite quaternion algebra, and denote by $\widehat{R}$ its adelization. Two left ideals\footnote{I.e., an additive subgroup stable under multiplication and of $\calo_F$-rank $4$.} $I,J$ of $R$ belong to the same class if $Ib=J$ for some $b\in B^\times$. The set of all such classes is denoted by $\Cl(B)$. It admits an interpretation as an adelic double quotient (see \cite[p.87-88]{vign}) given by the following bijection
\[
\Cl(B)\simeq G(\Q)\backslash G(\A_f)/\widehat{R}^\times.
\]
By strong approximation, we have that $\Cl(B)$ is a finite set.

\subsubsection{Gross Curve}
To $B$ we associate a curve in the following way (see \cite[p.418]{bd} for more details). 
\\Denote by $\mathscr{P}$ the conic curve over $\Q$ defined by
\[
\mathscr{P}(A)=\{x\in B\otimes A:\;x\neq0,\;\nr(x)=tr(x)=0\}/A^\times
\]
for a $\Q$-algebra $A$. Note that the $K$-points of $\mathscr{P}$ are naturally in bijection with $\Hom(K,B)$. The $\Q$-algebraic group $\Aut(\mathscr{P})=G(\Q)$ acts by conjugation on $\mathscr{P}$, and we obtain the $\mathit{Gross}$ curve of level $R$
\[
X_R=G(\Q)\backslash\mathscr{P}\times G(\A_f)/\widehat{R}^\times.
\]
For $(g_i)_{i\le r}$ a set of representatives of $\Cl(B)$, we set $\Gamma_i=g_i \widehat{R} g_i^{-1}\cap B^\times$, which are finite subgroups of $B^\times$. We can therefore write the Gross curve
\[
X_R=\bigsqcup_{i=1}^r \Gamma_i\backslash \mathscr{P}
\]
as a disjoint union of conics over $\Q$.

\subsubsection{Gross Points}\label{grosspoints}
In the very influential paper \cite{gr}, Gross gave a geometric interpretation to the arithmetic of $\mathit{definite}$ quaternion algebras as follows.
\\

By the Skolem-Noether theorem we have the following bijection
\[
\Emb(K,B)\simeq K^\times\backslash B^\times.
\]
We define $\tau\colon K\hookrightarrow B$ to be an $\mathit{optimal}$ embedding if it maps the $\calo_F$-order $\calo$ in $K$ to the $\calo_F$-order $R$ of $B$ and does not extend to an embedding of any larger order into $R$. In symbols, $\tau(K)\cap R=\tau(\calo)$. Denote the set of these optimal embeddings by $\Opt(\calo,R)$
Note that being optimal is a local property (see \cite[Lemma 30.3.6]{voi}).

For $c\ge 0$, we recall that $\calo_c:=\calo_F + c\calo_K$. Let $\tau\colon\calo_c\hookrightarrow R$ be an optimal embedding. A $\mathit{Gross}$ $\mathit{point}$ of conductor $c$ is the equivalence class of the pair $(R,\tau)$ by $B^\times$-conjugation. In other words, Gross points on $X_R$ over $K$ are the image of $\mathscr{P}(K)\times G(\A_f)/ \widehat{R}^\times$ in $X_R(K)$. Denote by $\Gr(c)$ the set of Gross points of conductor $c$. There is moreover a simply transitively action of $\widehat{\calo}_n\backslash T(\A_f)/ T(\Q)$ on the set of Gross points of conductor $c$ (see \cite[Lemma 2.5]{bd}). 
\\

Consider the $\Q$-algebraic groups $G=\text{Res}_{F/\Q}(B^\times)$ and $T=\text{Res}_{F/\Q}(K^\times)$. Gross points corresponds to (discrete) set of CM points of level $R$
\[
\text{CM}_R:=T(\Q)\backslash G(\A_f)/\widehat{R}^\times
\]
considered in \cite[p.7]{cv} in the definite case. Moreover, as explained in \cite[p.419]{bd}, one has that $\mathscr{P}$ coincides with $\Hom(K_v, B_v)$ at the level of $K_v$-points.

\begin{rmk}
	We remark that Gross points play the role of CM points in the Shimura curve when the quaternion algebra is definite, which in this case form simply an infinite discrete set.
	Despite this, there is a caveat: although sometimes called Heegner points (as in \cite{vat}), Gross points are not the CM points in the classical sense for an indefinite quaternion algebra. 
\end{rmk}

On the other hand, if $\tau$ is not optimal, then there exists a $\calo'$ containing $\calo$ such that it is optimal for $\calo'$. Thus the set $\Emb(\calo,R)$  can be decomposed as the disjoint union of sets $\Opt(\calo',R)$ of optimal embeddings, that is,
\[
\Emb(\calo,R)=\coprod_{\calo\subseteq\calo'} \Opt(\calo',R)
\]
Moreover, for  
$E_{\calo,R}:=\{b\in B^\times : b^{-1}Kb\cap R=b^{-1}\calo b \}=\{b\in B^\times : K\cap b^{-1}Rb=\calo\}$, we have the following bijection
\[
\Opt(\calo, R)\leftrightarrow K^\times \backslash  E_{\calo, R}
\]
given by $b\mapsto \tau_{b}$ such that $\tau_b(a)=b^{-1}ab$. Therefore
\[
\Emb(\calo_K,\calo_B)\leftrightarrow K^\times\backslash \{b\in B^\times: \calo_K\subseteq K\cap b^{-1}\calo_Bb\}.
\]
The previous set in terms on conductors $c$ can be rewritten as
\[
	\Emb(\calo_K,\calo_B)\leftrightarrow K^\times\backslash\coprod_c  \{b\in B^\times: \calo_c =K\cap b^{-1}\calo_Bb\}.
\]

\subsubsection{Auxiliary Spaces and Equidistribution}

Let $B_S$ be the quaternion algebra introduced in Section \ref{rm}, and consider the  $\Q$-algebraic groups $G_S=\text{Res}_{F/\Q}(B_S^\times)$ and $Z=\text{Res}_{F/\Q}(F^\times)$. 

We define the auxiliary profinite group
\[
G(S):=\prod'_{v\notin S}B_{S,v}^\times\times\prod_{v\in S}F^\times_v
\]
where the restricted product is considered with respect to the closure in $B_{S,v}$ of some fixed $\calo_{F}$-order in $B_S$. This is related to $G_S$ by the continuous map $\pi_S\colon G_S(\A_f)\twoheadrightarrow G(S)$, which is induced by the norm $\nr_{S,v}\colon B_{S,v}^\times\rightarrow F_v^\times$.

We thus obtain the finite sets of $\mathit{special}$ points (of level $R$) 
\[
\calx(S)_R=G(S,\Q)\backslash G(S)/\widehat{R}^\times
\]
and of $\mathit{connected}$ $\mathit{components}$
\[
\calz_R=Z(\Q)^+\backslash Z(\A_f)/\widehat{R}^\times,
\]
where $G(S,\Q)=\pi_S(G_S(\Q))$.

Lastly, the $\mathit{reduction}$ map $red_S\colon\scal^{\text{CM}}_R\rightarrow \calx(S)_R$ is induced by the map $\phi_S\colon G(\A_f)\rightarrow G(S)$ as defined in \cite[p.6]{cv} and the $\mathit{connected}$ $\mathit{component}$ map $c_S\colon \calx(S)_R\rightarrow\calz_R$ is induced by $\nr_{S,v}$.

\begin{thm}
	Let $v$ ramify in $K$. Denote $\Gr(\infty)=\bigcup_{c\in\N}\Gr(c)$. Let also $\calg$ be as in Theorem \ref{mainthm} and suppose that $s\in c_v^{-1}(\calg x)$. Then
	\[
	\lim_{x\to\infty} \dfrac{\#\{x\in\Gr(\infty)\;|\; red_v(x)=s \}}{\#\calg x}=O_\calg^{-1}I(s) 
	\]
	where $O_\calg^{-1}I(s)$ is defined in (\ref{equiram}).
\end{thm}

As mentioned in the introduction, a (simultaneous) equidistribution result follows mutatis mutandis from the same proof of the indefinite case.

\end{document}